\documentclass[a4paper,11pt]{article}
\usepackage[plainpages=false]{hyperref}
\usepackage{amsfonts,latexsym,rawfonts,amsmath,amssymb,amsthm,mathrsfs}
\usepackage{amsmath,amssymb,amsfonts,latexsym,lscape,rawfonts}

\usepackage[all]{xy}
\usepackage{eufrak}
\usepackage{makeidx}         
\usepackage{graphicx,psfrag}

\usepackage{array,tabularx}

\usepackage{setspace}

\newtheorem{thm}{Theorem}[section]

\newtheorem{lem}[thm]{Lemma}
\newtheorem{clm}[thm]{Claim}
\newtheorem{prop}[thm]{Proposition}

\theoremstyle{remark}
\newtheorem{rmk}[thm]{Remark}

\theoremstyle{definition}
\newtheorem{Def}[thm]{Definition}                                        %

\def \C {\mathbb C}

\title{On the long time behaviour of the Conical K\"ahler-
Ricci flows}
\author{Xiuxiong Chen,\ Yuanqi Wang}
\begin{document}

\maketitle{}
\begin{abstract}
We prove that the conical K\"ahler-Ricci flows introduced in \cite{CYW} exist for all time $t\in [0,+\infty)$. These immortal  flows possess maximal regularity in the conical category. As an application, we show if the twisted first Chern class $C_{1,\beta}$ is negative or zero, the corresponding conical K\"ahler-Ricci flows converge to    K\"ahler-Einstein metrics with conical singularities exponentially fast.  To establish these results,   one of our key steps is to prove  a Liouville type theorem for   K\"ahler-Ricci flat metrics (which are defined over $\mathbb{C}^{n}$) with conical singularities.  
\end{abstract}
\tableofcontents
\section{Introduction}
Let $(M, [\omega_0])$ be a polarized K\"ahler manifold and $D$ is a smooth divisor of the anti-canonical line bundle.  Suppose the  ``twisted" first Chern class
($\beta \in (0,1)$ )\[
C_{1,\beta}  =     C_1(M) -  (1-\beta) C_{1}[D]
\] 
has a definite sign.  One important question is to study the existence of the conical K\"ahler-Einstein metric in $(M, [\omega_0], (1-\beta)[D])$
\[
Ric(\omega_\phi) =  \beta \omega_\phi + 2\pi (1-\beta)[D].
\]
This problem has been studied carefully by many authors, for instance, \cite{Berman}, \cite{Brendle}, \cite{JMR},\cite{SongWang},\cite{LiSun} etc.
In particular, ``conical K\"ahler-Einstein metric" is a key ingredient in the recent solution of existence problem for K\"ahler-Einstein metric with positive scalar curvature
\cite{CDS1}\cite{CDS2}\cite{CDS3}. In light of these exciting development, we introduce the notion of   conical K\"ahler-Ricci flow  in \cite{CYW}
\begin{equation}\label{Definition of CKRF}
{{\partial \omega_g}\over {\partial t}}  = \beta\omega_g -   Ric(g)+ 2\pi(1-\beta)[D],
\end{equation}
 to attack the existence problem of
conical K\"ahler-Einstein metrics and conical K\"ahler-Ricci solitons.  In \cite{CYW}, we establish short time existence for this flow initiated from any $(\alpha,\beta)$ conical K\"ahler metric (see Section \ref{section of C0 estimate and C11 estimate} for the definition of  $(\alpha,\beta)$ metrics while we
follow the notations in \cite{CYW} in general). 
This is the second paper in this series where we want to establish the long time existence of this flow.  
\begin{thm}\label{long time existence of CKRF over Riemann surface}  Suppose  $g_0$ is an $(\acute{\alpha},\beta)$-conical K\"ahler metric in $(M, (1-\beta) D)\;$ where $\alpha' \in  (0, \min\{\frac{1}{\beta}-1,1\}).\;$
Then the conical K\"ahler-Ricci flow equation  admits a solution $\phi(t)$ (\ref{Definition of CKRF}) for  $t\in [0,\ +\infty)$. Furthermore, we have 
 
\begin{itemize}
       \item  for every $t\neq 0$, $g(t)$ is  an $(\alpha,\beta)$-conical metric in $(M, (1-\beta) D);\;$ for all $\alpha<\min\{\frac{1}{\beta}-1,1\}$;
       \item for all $N>1$, over the time interval $[0,N]$, $g(t)$ is a $C^{\alpha,\frac{\alpha}{2},\beta}[0,N]$-family of conical  metrics (for all $\alpha<\min\{\frac{1}{\beta}-1,1\}$).
      \end{itemize}
      
\end{thm}
\begin{rmk}The conical flow $\phi(t)$ in Theorem \ref{long time existence of CKRF over Riemann surface} possess $C^{2,\alpha,\beta}$-regularity, while the weak-flow $\phi(t)$ constructed in \cite{WYQWF} only possess 
$C^{1,1}$-regularity apriorily (so the metric tensor is not $C^{\alpha,\beta}$ aprorily). This is the essential difference between (strong) conical flow and weak-conical flow. Theorem \ref{long time existence of CKRF over Riemann surface} actually implies the weak flow constructed in \cite{WYQWF} is strong. Along the line of weak conical flows,  in \cite{LiuZhang}, Liu-Zhang also construct  weak conical  flows and obtain  convergence results of  their flows on Fano manifolds when $\beta\leq \frac{1}{2}$. 
\end{rmk}

\begin{rmk} For smooth K\"ahler-Ricci flow, the global existence of flow is proved by Cao  \cite{Cao}.  For conical K\"ahler-Ricci flow,  when $n=1$, this is recently proved
by Yin \cite{Yin2} and independently by  Mazzeo-Rubinstein-Sesum \cite{MRS} with different functional spaces. 
\end{rmk}

\begin{rmk} For simplicity, we only present the case with one smooth divisor. Our proof certainly works with reducible smooth divisors with no self intersections and with possibly different angles along each component. 
\end{rmk}

\begin{rmk} If the manifold is not Fano  or  the twisted first Chern has mixed sign,  Theorem 1.1 still holds as long as the
evolving K\"ahler class remains to be a K\"ahler class. In particular, the flow is immortal if it fixes the K\"ahler  class.
\end{rmk}

\begin{rmk}  This theorem may leads to some exciting, plausible future research: to ``migrate" a network of important, fundamental results established
in smooth K\"ahler-Ricci flow to our settings.   A partial list of these works (which is far from complete) is  given below and  we refer interested readers to these papers and  references therein for further readings:
  \cite{QZhang1} \cite{Ye} \cite{CWB1} \cite{QZhang2};  \cite{TZ1}\cite{SongTian}\cite{SongWeinkove} \cite{PhongSturm}etc.   A word of caution is, because of the presence of conical singularities, that this ``migration" might not be at all straightforward!  \end{rmk}
As an almost direct application,  the following is true.
\begin{thm}\label{Convergence to KE metric when C1 nonpositive} If $C_{1,\beta} < 0$ or  $C_{1,\beta} = 0$, then the corresponding conical K\"ahler-Ricci flow converges exponentially fast to
a conical K\"ahler-Einstein metric in the $C^{\alpha,\beta}_{1,1}$ topology of $(1,1)$-forms (in the sense of (\ref{Calabi-Yau: Exponential convergence of the metric in the Calpha beta 1,1 topology.})).
\end{thm}

\begin{rmk} When $C_{1,\beta} = 0, \beta \leq {1\over 2}$, the existence of Ricci-Flat conical metrics  is due to S. Brendle \cite{Brendle} via continuity method. When $C_{1,\beta} < 0,$
the existence has been studied by via continuous methods by Jeffres-Mazzeo-Rubinstein \cite{JMR}, Campana-Guenancia-Paun \cite{CGP}, and Eyssidieux-Guedj-Zeriahi \cite{EGZ}.
\end{rmk}
\begin{rmk} In the work of Li-Sun \cite{LiSun}, they consider the log Calabi-Yau pair \[(X, \sum_{i=1}^N \; (1-\beta_i) D_i) \] such that
\[
   C_1(X) - \sum_{i=1}^N \; (1-\beta_i) C_1(D_i)  = 0.
\]
Then, Theorem \ref{Convergence to KE metric when C1 nonpositive} implies that the existence of Calabi-Yau metric with correct cone angle for any log 
Calabi-Yau pair. It seems that the existence result for log Calabi-Yau pair of this generality is new. For related topics, please see Song-Wang's work \cite{SongWang}.

\end{rmk}
\begin{rmk}In Cao's proof \cite{Cao} on the smooth case, the Li-Yau harnack inequality in \cite{LiYau} plays a key role when showing the limit is K\"ahler-Einstein when $C_{1,\beta}=0$. In our conical case, it's not clear to us whether the Li-Yau type estimates hold. In our case, the monotonicity of the K-energy directly implies the limit is  K\"ahler-Einstein and the convergence of the metric tensor is exponential. 
\end{rmk}
   Going back to Theorem 1.1, much like the smooth counter part, we need to prove
   $C^0$-estimate of evolved potentials.  First, one needs to reduce the flow into a scalar equation. Suppose  $\omega_D$ is  the model conical K\"ahler metric (defined in \cite{Don}, also see the introduction of \cite{CYW})
   with cone angle $\beta$ over $D$ and $h_{\omega_D}$ denotes its Ricci potential.
   Then, 
\begin{equation}\label{ckrf potential equation}
 \frac{\partial \phi}{\partial t}=\log\frac{(\omega_D+\sqrt{-1}\partial\bar{\partial}\phi)^n}{\omega_D^n}+\beta \phi+ h_{\omega_D}.
\end{equation}
Routine calculation shows that $h_{\omega_D} \in C^{\alpha, \beta}$ for some $\alpha > 0. $
   \begin{prop}\label{C0 estimate0} Suppose the conical K\"ahler-Ricci flow exists up to time $T > 0.\;$ Then, there exists a uniform constant $C_T$ such
   that
   \[
        |\phi| + |{{\partial \phi}\over {\partial t}}|(t) < C_T, \qquad \textrm{forall}\  t\in [0,T).
   \]
   \end{prop} 
   Following \cite{JMR} and \cite{CDS2} (elliptic case), we can use a parabolic type Chern-Lu inequality to obtain:
   \begin{prop} \label{C^2 estimate0} Under the same assumptions as in Prop \ref{C0 estimate0},   we have $$\frac{1}{K}\omega\leq \omega+\sqrt{-1}\partial\bar{\partial}\phi\leq K\omega.$$
\end{prop}
\begin{rmk}  Guenancia-Paun's trick in \cite{GP} also works well for the $C^{1,1}$-estimate here. Actually, we have multiple choices here to prove the $C^{1,1}$-estimate.
\end{rmk}

 To prove long time existence, we essentially need to prove a priori Holder estimate for the evolving conical K\"ahler forms.   A critical step  for this type estimate
 is to prove the following Liouville type theorem:

\begin{thm}\label{thm Liouville}(Liouville Theorem) Suppose $\omega$ is a $C^{\alpha,\beta}$ conical K\"ahler  metric defined over $C^{n}$. Suppose there is a constant  $K$ such that  
\begin{equation}
\omega^{n}=\omega_{\beta}^{n},\ \frac{1}{K}\omega_{\beta}\leq \omega\leq K\omega_{\beta}\ \textrm{over}\ \C \times \C^{n-1}\setminus \{z=0\}.
\end{equation}
Then, there is a linear transformation $L$ which preserves $\{z=0\}$ and \[\omega=L^{\star}\omega_{\beta}.\]
\end{thm}

This plays a central role in the proof of long time existence theorem.   When conical singularity is not presence, this is due to  Riebesehl- Schulz \cite{RS84} where higher derivatives are used heavily. The problem certainly goes back to the famous
paper by E. Calabi \cite{Calabi58} and Pogorelov \cite{Po78}. Even in the smooth setting, this is considered an alternative approaches to the later famous
Evans-Krylov Shauder estimate for Monge-Ampere equation (cf. \cite{Evans}\cite{GT}).\\

  To prove this Liouville type theorem, we need to extend the maximal principle to more general settings.  In the literature, it seems to be a standard trick to use Jeffery's trick whenever we need to apply Maximum principle. A standard feature of the Jeffery's trick is to
add a small copy of small power of $|S|$ where $S$ is the defining holomorphic section of divisor;  and this will perturb the maximum point off from divisor, which allows
us to use standard maximal principle.  For this trick to work,  an important pre-condition is
that the function, which we applied maximum principle to,  must be $C^{\alpha,\beta}$ for some $\alpha > 0$. This restricts severely how we can use maximum principle. 
In this paper, we are able to remove this restriction and  are able to adapt both weak and strong maximal principle to our setting for function
which is  locally smooth away from divisor and $L^\infty$ globally. Indeed, we plan to apply maximum principle to $tr_{\omega_{\beta}}\omega_\phi$
which can only be $L^\infty$ globally.  


      \begin{thm}\label{Evans Krylov}Suppose $g(t),\ t\in[0,T)$ is a solution to the CKRF in Theorem \ref{long time existence of CKRF over Riemann surface}, $T<\infty$. Then there exists constant $\mathbb{K}$ in the sense of Def \ref{Convention on the constant}  such that the potential $\phi$ satisfies the following bound \[|\phi|_{2,\alpha,\beta}\leq \mathbb{K}\ \textrm{for all}\ t\in [0,T).\] 
      
Consequently, the flow $g(t),\ t\in[0,T)$ can be extended beyond $T$.
    \begin{rmk} For conical K\"ahler Einstein metric, the corresponding a priori estimate is  derived in \cite{CDS2} (c.f. discussions in \cite{JMR}).  This parabolic type Holder estimate should be able to extend as an a priori  $C^{2,\alpha,\beta}$-estimate for continuity method for solving the K\"ahler-Einstein equations  (as in \cite{LiSun} and \cite{JMR}).
    \end{rmk}   

\end{thm}

 One of the key ingredients of Theorem \ref{thm Liouville} is the theory of weak solutions to the Laplace equation of a concial metric. Fortunately, 
 in the polar coordinates, a cone metric is quasi-isometric to the Euclidean metric away from $D$ (by definition). Thus, though straight forward to observe, it's suprising and amazing that  the weak-solution theory in Chap III of \cite{LSU} (De-Giorge estimate), Chap 8 of \cite{GT}, and Chap 4 of \cite{HanLin} are all directly  applicable. Roughly speaking, this is because  the weak-solution theory only involves  $W^{1,2}$-quantities of  the weak solution. Thus after integration by parts, we can transform the 
 $W^{1,2}$-inequalities with respect to the cone metric $\omega$ to
 $W^{1,2}$-inequalities with respect to the Euclidean metric $g_{E}$ 
 (in the polar coordinates)! Thus, all the classical tools can be applied. Though in most of the place we directly use weak solution theory (Moser's iteration, Weak harnack inequalities...), we still prove Trudinger's Harnack inequality in detail in Appendix B, to show how to   
 transform the 
 $W^{1,2}$-inequalities with respect to the cone metric $\omega$ to
 $W^{1,2}$-inequalities with respect to the Euclidean metric, and then directly apply the results in \cite{LSU}, \cite{GT}, and \cite{HanLin}.
 This proof, though straight forward,   shows  all the required estimates in weak solution theory  are true in our situation.

  Strategy of our work and organization of this article: In Section \ref{section of C0 estimate and C11 estimate} we state some conventions of notations and prove the $C^0$ and $C^{1,1}$ estimate. Then we  study the $C^{2,\alpha,\beta}$ regularity  in Section 3---11. In Section \ref{section of Holder estimate for the second derivative}, assuming Theorem \ref{thm Liouville},  we prove Theorem \ref{Evans Krylov} by showing the H\"older radius is uniformly bounded, thus settle down the $C^{2,\alpha,\beta}$ estimate and the proof of Theorem \ref{long time existence of CKRF over Riemann surface}. In Section \ref{section Poincare Lelong equations} we solve the Poincare-Lelong equation with the correct estimates, which is crucial when we perturb the rescaled limit back to get a contradiction. In Section \ref{section A rigidity theorem}---\ref{section Strong Maximum Principles and De Giorgi estimate}, we establish the analytic tools for proving Theorem \ref{thm Liouville}. In section \ref{section Proof of Theorem liouville} we prove Theorem \ref{thm Liouville}.  In Section \ref{section Bootstrapping of  CKRF} we prove that the flow has maximal regularity immediately when $t>0$, which is crucial when proving the convergence of the rescaled flows. In Section \ref{section Exponential Convergence when C_1 nonpositive} we prove when $C_{1,\beta}<0$ or $=0$, the CKRF converges to conical K\"ahler-Einstein metric exponentially. In Appendix A, we present a short proof of the Liouville theorem in the case when $\beta\leq \frac{1}{2}$, where we use a regularity result due to Brendle \cite{Brendle}. In Appendix B, we prove Trudinger's  Harnack inequality in our case, by directly using the results from \cite{GT} and \cite{HanLin}.\\
  
  Acknowledgment:    The first named author wish to thank  Kai Zheng,  Chengjiang Yao for helpful discussions on maximal principle over conical settings.  He is also grateful to  Xi Zhang
for sharing his insight on weak conical-K\"ahler Ricci flow. The second named author wish to thank Prof S.K Donaldson for communications on earlier versions of this work. The second author is  grateful to Prof Xianzhe Dai, Guofang Wei, and Rugang Ye for their interest in this work and their continuous support. He also would like to thank  Yuan Yuan for related  discussions on  complex analysis. Both authors would like to thank   Song Sun, Kai Zheng, and Haozhao Li for carefully reading earlier versions of this paper and related discussions. 

 \section{Convention of notations,  $C^0$ estimate, and $C^{1,1}$ estimate.\label{section of C0 estimate and C11 estimate}}

\begin{Def}\label{Convention on the constant}  We would like to make following convention on the constants in this paper, similar to that of \cite{CYW}:
   \textit{Without further notice, the  "C" in each estimate  means a constant depending on the dimension $n$, the angle $\beta$, the background objects $(M,\omega_{0},L,h,D,\omega_{D})$, the $\alpha$ (and $\acute{\alpha}$ if any) in the same estimate or in the corresponding theorem (proposition, corollary, lemma), the initial metric $\omega_{0}$, and finally the time $T$ ( beyond which we want to extend the flow). We add index to the "C" if it depends on more factors than the above objects. Moreover, the "C" in different places might be different. The $C$ never depends on $T^{\prime}<T$ (unless it comes with another index, like $C(T^{\prime})$).}
   \end{Def}
   
   Most of the notations in this article follow those of \cite{CYW}. For the readers' convenience, we introduce some key definitions from \cite{CYW} here.

\begin{Def} \label{definition of (alpha,beta) conical metric:single divisor.} $(\alpha, \beta)$ conical K\"ahler metric:\  For any $\alpha \in (0, \min\{{1\over \beta}-1,1\}),\;$ a   K\"ahler form $\omega$ is said to be an $(\alpha, \beta)$  conical K\"ahler metric on $(M, (1-\beta) D)$  if it satisfies the following conditions.
\begin{enumerate}
\item $\omega$ is a closed positive $(1,1)$-current over $M$.
\item For any point $p\in D$, there exists a  holomorphic chart $\{z,u_i,\ i=1,..,n-1\}$
such that in this chart, $\omega$ is quasi-isometric to the standard cone metric
\[
 \omega_{\beta}=\beta^2|z|^{2\beta - 2} \frac{\sqrt{-1}}{2}d z \wedge d \bar z + {\sqrt{-1}\over 2}  \displaystyle \sum_{j=2}^{n} d u_j \wedge d \bar u_j.
\]
\item There is a $\phi\in C^{2,\alpha,\beta}(M)$  such that \[\omega=\omega_0+ i \partial \bar \partial \phi.\]  
\end{enumerate}
\end{Def}
\begin{rmk} If an $\omega$, defined either globally or locally,  satisfies $1$ and $2$ in Definition \ref{definition of (alpha,beta) conical metric:single divisor.}, but only partially satisfies $3$ in the sense that $\omega=\sqrt{-1} \partial \bar \partial \phi$ for some $\phi\in C^{\alpha,\beta}$, then we say $\omega$ is a weak conical metric.
This definition can be found in Definition 1.2 in \cite{WYQWF}.  
\end{rmk}

Near $D$, using the defining function $z$ of $D$, it's easy to describe the function space $C^{2,\alpha,\beta}(M)$ near $D$. Namely, near $D$, let $\xi$ be singular coordinate, $z=|\xi|^{\frac{1}{\beta}-1}\xi.$ Let $z=\rho e^{i\theta}$, and $s_i,i=3....2n$ be real coordinates of $z_2,...z_{n}$ which are perpendicular to $z$. We define

\begin{Def}($C^{2,\alpha,\beta}$-functions).
\begin{enumerate}
  \item $f(z,z_2.....)\in C^{\alpha,\beta}$ iff $f(|\xi|^{\frac{1}{\beta}-1}\xi,z_2...)\in C^{\alpha}$ in terms of $\xi,z_2...$
  \item $f(z,z_2.....)\in C^{2,\alpha,\beta}$ iff
  $$|z|^{2-2\beta}\frac{\partial^2f}{\partial z\partial \bar{z}}\in C^{\alpha,\beta},|z|^{1-\beta}\frac{\partial^2f}{\partial \rho\partial s_i}\in C^{\alpha,\beta},$$
    $$|z|^{-\beta}\frac{\partial^2f}{\partial \theta\partial s_i}\in C^{\alpha,\beta
    },\frac{\partial^2f}{\partial s_i\partial s_j}\in C^{\alpha,\beta
    }.$$
   \end{enumerate} 
   The
full definition of the function space $C^{2,\alpha,\beta}(M)$ and the corresponding parabolic norm  is in section 2 of \cite{CYW}.

\end{Def}

 The following model metric defined in \cite{Don} satisfies the above definition.
 \begin{equation*}\label{Model conical metric defined by Donaldson}\omega_{D}=\omega_{0}+\delta i \partial \bar \partial |S|^{2\beta}
 ,\ \textrm{where}\ \delta \ \textrm{is a small enough number}.\end{equation*}

Let $r=|z|^{\frac{1}{\beta}}$ and $\theta$ be just  the angle of $z$ from the positive real axis. In the polar coordinates $r,\theta, u_{i},\ 2\leq i\leq n$, $\omega_{\beta}$ can be written as 
$$\omega_{\beta}=dr^2+\beta^2r^2d\theta^2+\Sigma_{j=2}^{n}du\otimes d\bar{u}.$$
 Notice in the polar coordinates we have $\beta^2g_{E}\leq \omega_{\beta}\leq \frac{1}{\beta^2}g_{E}$, where $g_{E}$ is Euclidean metric in the polar coordinates i.e
   \begin{equation}\label{equ def of g_E in polar coordinates}g_{E}=dr^2+r^2d\theta^2+\Sigma_{j=2}^{n}du\otimes d\bar{u}.\end{equation} 
   From now on we will be using the polar coordinates in most of the sections, since there the conical metrics are  quasi-isometric to the Euclidean metric.

 Theorem \ref{long time existence of CKRF over Riemann surface} is a direct consequence of Proposition \ref{Classical solution's C^0 estimate}, \ref{C^2 estimate}, and Theorem \ref{Evans Krylov}. Now we prove the $C^0$ estimate. 
 \begin{prop}\label{Classical solution's C^0 estimate}Suppose the flow exists over $[0,T^{\prime}]$. Then we have the following $C^{0}$ bound  \[|\phi|_{0}\leq a+\frac{b}{\beta}e^{\beta t},\ \textrm{for all} \ (x,t)\in (M\setminus D)\times [0,T^{\prime}],\] where $a=\sup_{x}|\phi|(x,0)$, and $b=\sup_{x}\{|\log\frac{(\omega_{D}+\sqrt{-1}\partial\bar{\partial}\phi_0)^n}{\omega_{D}^n}+\beta\phi+h_{\omega_{D}}|(x,0)\}$. 
 \end{prop}Over $M\setminus D$, denote $u=\frac{\partial \phi}{\partial t}$,  we compute \begin{equation} \label{parabolic equation of the Ricci potential}
 \frac{\partial u}{\partial t}=\Delta_tu+\beta u.
 \end{equation}
We use the function $|S|^{2\tau}$ as barrier function such that  $2\tau\leq \beta\alpha$. Suppose the flow is smooth over $[0,T^{\prime}]$. First we have the following lemma due to Jeffres \cite{Jeffres}.
\begin{lem}\label{maximal attained away from divisor}$u+\epsilon |S|^{2\tau}$ attains maximum in $M\setminus D$ when $t\in [0,T^{\prime}]$; $u-\epsilon |S|^{2\tau}$  attains minimum in $M\setminus D$.
\end{lem}
For the reader's convenience of include the proof here. 
\begin{proof}{of Lemma \ref{maximal attained away from divisor}:}\  It suffices to prove $u+\epsilon |S|^{2\tau}$ attains maximum in $M\setminus D$ when $t\in [0,T^{\prime}]$, the other is similar. We argue by contradiction. If not, suppose $u+\epsilon |S|^{2\tau}$ attains maximum in $D$ at $(p,\ t_1)$.
Let $\rho=|z|$, notice that the integral curve of $\frac{\partial}{\partial \rho}$ is not necessarily the geodesic. Then let $q$ be on the  integral curve $\gamma$ of $\frac{\partial}{\partial \rho}$ starting at $p$. We have
\begin{equation}\label{C0C11 contradiction if power of |S| is small} \frac{|u(q)-u(p)|}{\rho^{\beta\alpha}(q)}\leq |u|_{\alpha}< \infty;\ \frac{u(q)-u(p)+\epsilon|S|^{2\tau}(q)}{\rho^{\beta\alpha}(q)}\leq 0.\end{equation}

Since $2\tau-\beta\alpha<0$,  when $\rho(q)$ is sufficiently small with respect to $|u|_{\alpha}$, we have
\begin{equation}\label{S norm has big slope}\frac{\epsilon |S|^{2\tau}}{\rho^{\beta\alpha}(q)}\geq C\epsilon \rho^{2\tau-\beta\alpha}(q)\geq 2|u|_{\alpha}+1.
\end{equation}
Therefore (\ref{S norm has big slope}) contradicts (\ref{C0C11 contradiction if power of |S| is small}).
\end{proof}
\begin{proof}{of Proposition \ref{Classical solution's C^0 estimate}:}\ We compute $$\frac{\partial (u+\epsilon|S|^{2\tau})}{\partial t}=\Delta_t(u+\epsilon|S|^{2\tau})+\beta (u+\epsilon|S|^{2\tau})-\Delta_t\epsilon|S|^{2\tau}-\beta \epsilon|S|^{2\tau}.$$
Using (31) in \cite{WYQWF}, we know  \begin{equation}\label{Subharmonicity of norm of S}
\Delta_t\epsilon|S|^{2\tau}\geq -C(T^{\prime})\epsilon ,
\end{equation}

then 
\begin{equation}\label{Perturbed Ricci potential satisfies a parabolic inequality}
\frac{\partial (u+\epsilon|S|^{2\tau})}{\partial t}\leq \Delta_t(u+\epsilon|S|^{2\tau})+\beta (u+\epsilon|S|^{2\tau})+C(T^{\prime})\epsilon.
\end{equation}
By Lemma \ref{maximal attained away from divisor}, the maximum-principle applies to (\ref{Perturbed Ricci potential satisfies a parabolic inequality}). Hence
$$(u+\epsilon|S|^{2\tau})\leq e^{\beta t}(|u+\epsilon|S|^{2\tau}|_{(0,t=0)})+C(T^{\prime})\epsilon$$
Let $\epsilon\rightarrow 0$ we get $u\leq e^{\beta t}|u|_{(0,t=0)}$. Thus the upper bound is obtained. The lower bound $u\geq -e^{\beta t}|u|_{(0,t=0)}$ follows similarly. Then 
\begin{equation}\label{Bounding the Ricci potential}
|\frac{\partial \phi}{\partial t}|=|u|=\leq e^{\beta t}|u|_{(0,t=0)}.
\end{equation}
By integrating (\ref{Bounding the Ricci potential}) and using 
\[\frac{\partial \phi}{\partial t}|_{t=0}=(\log\frac{(\omega_D+\sqrt{-1}\partial\bar{\partial}\phi)^n}{\omega_D^n}+\beta \phi+ h_{\omega_D})|_{t=0},\] we obtain the desired bound in Proposition \ref{Classical solution's C^0 estimate}.
\end{proof}

Our next objective  is to prove the following $C^{1,1}$ bound for conical KRF. We follow the approach in \cite{JMR} and \cite{CDS2}. Notice that Guenancia-Paun's trick in \cite{GP} also works for the
$C^{1,1}$-estimate here.  \begin{prop} \label{C^2 estimate} Under the same assumptions in Theorem \ref{long time existence of CKRF over Riemann surface}, there exists a uniform constant $K$ in the sense of Definition \ref{Convention on the constant}  such that $$\frac{1}{K}\omega_{D}\leq \omega_{D}+\sqrt{-1}\partial\bar{\partial}\phi\leq K\omega_{D}.$$
\end{prop}

Let $u=g^{i\bar{l}}h_{j\bar{k}}f^{j}_{i}f^{\bar{k}}_{\bar{l}}$, $f=id$ is treat as a harmonic map from $M$ to $M$ itself. Choose $z_{p}$ as normal coordinates of $g$ (with K\"ahler form $\omega$) at $x$, and let $l,\ i,$ be normal coordinate indexes of $\omega$ also. Then we compute
\begin{eqnarray*}& & \Delta_{\omega}u
\\&=&g^{l\bar{i}}_{,p\bar{p}}h_{j\bar{k}}f^{j}_{l}f^{\bar{k}}_{\bar{i}}+h_{j\bar{k},d\bar{m}}f^{j}_{i}f^{\bar{k}}_{\bar{i}}f^{d}_{p}f^{\bar{m}}_{\bar{p}}+h_{j\bar{k}}f^{j}_{ip}f^{\bar{k}}_{\bar{i}\bar{p}}
\\&+&h_{j\bar{k}}f^{j}_{ip\bar{p}}f^{\bar{k}}_{\bar{i}}+h_{j\bar{k}}f^{j}_{i}f^{\bar{k}}_{\bar{i}p\bar{p}}+h_{j\bar{k},d}f^{d}_{p}f^{j}_{i}f^{\bar{k}}_{\bar{i}\bar{p}}
\\&+&h_{j\bar{k},\bar{d}}f^{\bar{d}}_{\bar{p}}f^{j}_{i}f^{\bar{k}}_{\bar{i}p}+h_{j\bar{k},\bar{d}}f^{\bar{d}}_{\bar{p}}f^{j}_{ip}f^{\bar{k}}_{\bar{i}}+h_{j\bar{k},d}f^{d}_{p}f^{j}_{i\bar{p}}f^{\bar{k}}_{\bar{i}}.
\end{eqnarray*}
Choose $j,k,d,m$ as normal coordinate index of $h$, then
\begin{eqnarray*}& & \Delta_{\omega}u
\\&=&g^{l\bar{i}}_{,p\bar{p}}h_{j\bar{k}}f^{j}_{l}f^{\bar{k}}_{\bar{i}}+h_{j\bar{k},d\bar{m}}f^{j}_{i}f^{\bar{k}}_{\bar{i}}f^{d}_{p}f^{\bar{m}}_{\bar{p}}+h_{j\bar{k}}f^{j}_{ip}f^{\bar{k}}_{\bar{i}\bar{p}}
\\&=&R^{l\bar{i}}f^{j}_{l}f^{\bar{j}}_{\bar{i}}-R^{h}_{j\bar{k},d\bar{m}}f^{j}_{i}f^{\bar{k}}_{\bar{i}}f^{d}_{p}f^{\bar{m}}_{\bar{p}}+f^{j}_{ip}f^{\bar{j}}_{\bar{i}\bar{p}}
\end{eqnarray*}

  Set $h=\omega_{D}$. Thus along the K\"ahler-Ricci flow, using $R^{h}_{j\bar{k},d\bar{m}}\leq C_1I$ \\ (see Li-Rubinstein's appendix in \cite{JMR}),  and \[\frac{\partial}{\partial t}u=(R^{i\bar{l}}-\beta g^{i\bar{l}})h_{j\bar{k}}f_{i}^{j}f_{\bar{l}}^{\bar{k}}\ \textrm{over}\ M\setminus D,\]  we obtain
\begin{eqnarray*}& & (\Delta_{\omega}-\frac{\partial}{\partial t})u
\\&=&\beta f^{j}_{l}f^{\bar{j}}_{\bar{l}}-R^{h}_{j\bar{k},d\bar{m}}f^{j}_{i}f^{\bar{k}}_{\bar{i}}f^{d}_{p}f^{\bar{m}}_{\bar{p}}+f^{j}_{ip}f^{\bar{j}}_{\bar{i}\bar{p}}
\\&\geq& -C_1u^2+\beta u+f^{j}_{ip}f^{\bar{j}}_{\bar{i}\bar{p}}.
\end{eqnarray*}
By adding the weight $e^{\lambda \phi}u$  we compute
\begin{eqnarray*}& & (\Delta_{\omega}-\frac{\partial}{\partial t})e^{\lambda \phi}u
\\&\geq& \lambda e^{\lambda \phi}u(n-u)-C_1e^{\lambda \phi}u^2+\beta e^{\lambda \phi}u+e^{\lambda \phi}f^{j}_{ip}f^{\bar{j}}_{\bar{i}\bar{p}}+2\lambda e^{\lambda \phi}<\nabla_{\omega}\phi,\ \nabla_{\omega}u>
\\&-&\lambda u\frac{\partial \phi}{\partial t}e^{\lambda \phi}+\lambda^2u|\nabla_{\omega}\phi|^2e^{\lambda \phi}.
\end{eqnarray*}
Using the inequality $(\Sigma_ka_kb_k)^2\leq (\Sigma_ka_k^2)(\Sigma_lb_l^2)$ and the following estimate 
\begin{eqnarray*}& &|\nabla_{\omega}u|^2
\\&=&\Sigma_{i,k,p,s,t}f^{i}_{kp}f^{\bar{i}}_{\bar{k}}f^{\bar{s}}_{\bar{t}\bar{p}}f^{s}_{t}
\\&\leq & \Sigma_{p}\{(\Sigma_{i,k}|f^{i}_{kp}|^2)^{\frac{1}{2}}(\Sigma_{i,k}|f^{k}_{i}|^2)^{\frac{1}{2}}(\Sigma_{s,t}|f^{\bar{s}}_{\bar{t}\bar{p}}|^2)^{\frac{1}{2}}(\Sigma_{s,t}|f^{s}_{t}|^2)^{\frac{1}{2}}\}
\\&=&uf^{j}_{ip}f^{\bar{j}}_{\bar{i}\bar{p}},
\end{eqnarray*}

it's easy to see  $$e^{\lambda \phi}f^{j}_{ip}f^{\bar{j}}_{\bar{i}\bar{p}}+2\lambda e^{\lambda \phi}<\nabla_{\omega}\phi,\ \nabla_{\omega}u>+\lambda^2u|\nabla_{\omega}\phi|^2e^{\lambda \phi}\geq 0.$$
Thus let $C_2=C_1+1$ and $\lambda=-C_2$ we get
\begin{lem}\label{Chen-Lu Inequality}
\begin{eqnarray*}& & (\Delta_{\omega}-\frac{\partial}{\partial t})e^{-C_2\phi}u
\\&\geq&e^{-C_2\phi}u^2-Ce^{-C_2\phi}u+C_2u\frac{\partial \phi}{\partial t}e^{-C_2 \phi}.
\end{eqnarray*}
\end{lem}

Now we are ready to prove the $C^{1,1}$-estimate.

\begin{proof}{of Proposition \ref{C^2 estimate}:}
From (\ref{Chen-Lu Inequality}), we obtain
\begin{eqnarray*}& & (\Delta_{\phi}-\frac{\partial}{\partial t})[e^{-C_2\phi}u+\epsilon|S|^{2\tau}]
\\&\geq&e^{C_2\phi}[e^{-C_2\phi}u+\epsilon|S|^{2\tau}]^2-C[e^{-C_2\phi}u+\epsilon|S|^{2\tau}]\\&+&C_2[e^{-C_2 \phi}u+\epsilon|S|^{2\tau}]\frac{\partial \phi}{\partial t}
+C\epsilon|S|^{2\tau}-C_2\epsilon|S|^{2\tau}\frac{\partial \phi}{\partial t}\\&-&2u\epsilon|S|^{2\tau}
-e^{C_2\phi}(\epsilon|S|^{2\tau})^2+\Delta_{\phi}\epsilon|S|^{2\tau}.
\end{eqnarray*}
Again similar to the proof of Proposition \ref{Classical solution's C^0 estimate}, since $e^{-C_2\phi}u\in C^{\alpha,\beta}[0,T^{\prime}]$ , then $\max( e^{-C_2\phi}u+\epsilon|S|^{2\tau})$ is attained in $M\setminus D$ when  $\tau < \alpha \beta$. Using $\Delta_t\epsilon|S|^{2\tau}\geq -\epsilon C(T^{\prime})$ (see formula (31) in \cite{WYQWF}), (\ref{Bounding the Ricci potential}), Prposition \ref{Classical solution's C^0 estimate}, and maximum-principle,  we have the following inequality 
$$\{e^{-C_2\phi}u+\epsilon|S|^{2\tau}\}_{p}\leq \epsilon C(T^{\prime})+C\{e^{-C_2\phi}\}_{p},$$
 where $p$ is the maximum point  of $e^{-C_2\phi}u+\epsilon|S|^{2\tau}$.
Thus by taking $\epsilon\rightarrow 0$,  we end up with 
\begin{equation}\label{C1,1  derivative is bounded by oscillation}
u\leq  Ce^{C_2 osc \phi}. 
\end{equation} 

(\ref{C1,1  derivative is bounded by oscillation}) means the following. Suppose $z_i, i\in (1,...n).$ are the normal coordinates of the background metric $\omega$ at a general point $p$ such that it also diagonalize $\sqrt{-1}\partial\bar{\partial}\phi$ at $p$, we have
$$\Sigma_i\frac{1}{1+\phi_{i\bar{i}}}\leq C_.$$
Since $\phi$ satisfies the equation 
\[\frac{(\omega_{D}+\sqrt{-1}\partial\bar{\partial}\phi)^n}{\omega_{D}^n}=e^{\frac{\partial \phi}{\partial t}-h_{\omega_D}-\beta\phi}\] and we have $$|\frac{\partial \phi}{\partial t}|+|\phi|\leq C ,$$
 we obtain  $$\frac{1}{C}\omega_{D}\leq \omega_{D}+\sqrt{-1}\partial\bar{\partial}\phi\leq C\omega_{D}.$$
\end{proof}
At this point, actually we've arrived at a simple proof of the long time existence when the complex dimension is $1$, with the help of the Harnack inequality. 
\begin{prop} \label{Long time existence when n=1}When $n=1$, the long time existence (Theorem \ref{long time existence of CKRF over Riemann surface}) follows from Proposition \ref{Classical solution's C^0 estimate}, \ref{C^2 estimate}, and Theorem 4.2 in \cite{WYQWF}, without involving the proof of Theorem \ref{Evans Krylov} in the next section. 
\end{prop}
\begin{proof}{of Proposition \ref{Long time existence when n=1}:}
\ Proposition \ref{Classical solution's C^0 estimate}, Proposition  \ref{C^2 estimate}, and equation (\ref{ckrf potential equation}) say that the assumptions in Theorem 4.2 in \cite{WYQWF} are fulfilled. Thus  from  Theorem 4.2 in \cite{WYQWF},  there is a $\alpha>0$ such that  $|\frac{\partial \phi}{\partial t}|_{\frac{\alpha}{2},\alpha,\beta}\leq C,\ t\in [0,T^{\prime}]$ for any $T^{\prime}<T$. 
Since $n=1$, from the potential equation (\ref{ckrf potential equation}) we get $|\sqrt{-1}\partial \bar{\partial}\phi|_{\frac{\alpha}{2},\alpha,\beta}\leq C$, which says 
\begin{equation*}
|\phi|_{2,\alpha,\beta}\leq C\ \textrm{over}\ [0,T).
\end{equation*} 
By the discussions in Step 2 of the proof of Theorem \ref{Evans Krylov}, the flow can be extended beyond $T$. 
\end{proof}

\section{H\"older estimate for the second derivatives and proof of Theorem \ref{long time existence of CKRF over Riemann surface}.\label{section of Holder estimate for the second derivative}}
Based on Theorem \ref{thm Liouville}, we are able to prove Theorem               \ref{Evans Krylov}, which in turn implies our  main Theorem  \ref{long time existence of CKRF over Riemann surface} in an obvious way. 
Let us first introduce a new notion of H\"older radius, which is  motivated by the
Harmonic Radius in Anderson's work  \cite{Anderson}. 

From now on in this section, we work in the singular polar coordinates, unless otherwise specified. For the reader's convenience, we use  the main definitions from \cite{CYW}. Let $w_j,\ j=2 \cdots n $ be the tangential variables. We consider a basis of $(1,0)$ vectors as
   \begin{equation}\label{mathfrak a}\mathfrak{a}=\frac{1}{\sqrt{2}}(\frac{\partial }{\partial r}-\frac{\sqrt{-1}}{\beta r}\frac{\partial }{\partial \theta}), \frac{\partial }{\partial w_j}, j=2...n.
   \end{equation}
   Set $\xi=z^{\beta}=re^{i\beta\theta}$, notice that
   \[\frac{\partial^2}{\partial \xi \partial\bar{\xi}}=\frac{1}{4}[\frac{\partial^2  }{\partial r^2}+r^{-1}\frac{\partial  }{\partial r}+\frac{1}{\beta^2}r^{-2}\frac{\partial^2  }{\partial \theta^2}].\]
 In this  singular polar coordinates, 
  we define the polar $\sqrt{-1}\partial \bar{\partial}$-operator to be the operator with  the following basis.
   $$\frac{\partial^2}{\partial \xi \partial\bar{\xi}}, \mathfrak{a}\frac{\partial}{\partial \bar{w}_i},\bar{\mathfrak{a}}\frac{\partial}{\partial w_i},\frac{\partial^2}{\partial w_i \partial \bar{w}_j}, 2\leq i,j\leq n,$$
     By abuse of notation, the $"\sqrt{-1}\partial \bar{\partial}"$s  in the polar coordinates all mean  the polar $\sqrt{-1}\partial \bar{\partial}$-operator defined above.
     
     From now on, when we write "$[\ \cdot\ ]$", we mean  seminorm; when we write "$|\cdot|$", we mean norm (which  contain lower order terms). These definitions can be found in section 2 of \cite{CYW}.
\begin{rmk}\label{rmk omega beta has 0 ossilation in polar}In the polar coordinates, under the above basis, we have 
\[[\omega_{\beta}]_{\alpha,\beta}=0.\]
This means $\omega_{\beta}$ is a constant tensor.
\end{rmk}

\begin{Def}\label{Holder radius}H\"older radius: Let $K$ be  as in Proposition \ref{C^2 estimate}, let $\underline{K}$  and $\widehat{K}$ be two constants large enough. Given a point $p\in B_{0}(R)$ (in the polar coordinates),  and a $C^{\alpha,\beta}$-metric $\omega$ defined over $B_{0}(R)$, we define the H\"older radius $r_p$ of $\omega$ at a point $p\in B_{0}(R_{0})$  to be the largest radius (with respect to the Euclidean metric in the singular polar coordinates), such that there exists a potential $\phi$ in $B_{p}(r_p)$ which satisfies 
\begin{itemize}
\item $\omega=\sqrt{-1}\partial \bar{\partial}\phi$ over $B_{p}(r_{p})$,
$r_{p}\leq d_{\beta,E}(p,\partial B_{0}(R))$.
\item $\phi\in C^{2,\alpha,\beta}$, $[\phi]_{2,\alpha,\beta}\leq \delta_{0}r_{p}^{-\alpha}$,\ $[\phi]_{2,\beta}\leq \underline{K}$,\ $|\phi|_{0}\leq \widehat{K}r_p^2$,
\end{itemize}
where  $\delta_{0}$ is small enough with respect to  the $\delta$ in  Proposition \ref{prop Gap}. For the second item, the norms are defined in the polar coordinates, as in section 2 in \cite{CYW}. The balls are all with respect to  $d_{\beta,E}$, which  is the distance with respect to the Euclidean metric $g_{E}$ in the polar coordinates. 
\end{Def}




\begin{proof}{of Theorem \ref{Evans Krylov} and \ref{long time existence of CKRF over Riemann surface}:}

Step 1:  By the $C^{1,1}-$estimate in Proposition \ref{C^2 estimate},  using Theorem 4.2 in \cite{WYQWF} and equation (\ref{parabolic equation of the Ricci potential}), we deduce
   \begin{equation}\label{equ Bound on the hold norm of the Ricci potential}|\frac{\partial \phi}{\partial t}|_{\alpha^{\prime},\frac{\alpha^{\prime}}{2},\beta,[0,T)}\leq C,\ \textrm{for some}\        \alpha^{\prime}>0.\end{equation}
    Moreover,  by Theorem 4.2 in \cite{WYQWF}, the $C^{0}$-estimate in Proposition \ref{Classical solution's C^0 estimate}, and the $C^{1,1}$-estimate in Proposition \ref{C^2 estimate}, and (\ref{equ Bound on the hold norm of the Ricci potential}), we obtain 
     \begin{equation}\label{equ Bound on the hold norm of phi}|\phi|_{\alpha^{\prime},\frac{\alpha^{\prime}}{2},\beta,[0,T)}\leq C.\end{equation}
    by making $\alpha^{\prime}$ smaller if necessary.

   Step 2. In this step we show  $|\phi|_{2,\alpha,\beta,[0,T)}$ is uniformly bounded, for any $\alpha<\alpha^{\prime}$. We follow the  Anderson-type  argument as in the proof of Lemma 2.2 in \cite{Anderson}.
By abuse a notation, we still denote $\phi$ as  the potential of $\omega$ near $D$ i.e $\omega=\sqrt{-1}\partial \bar{\partial} \phi$.

  Denote $\omega_{i}=\omega(t_{i})$, $t_{i}\in [0,T)$ is a time sequence. Denote $$F_{i}=(\frac{\partial \phi}{\partial t}-\beta\phi+f)|_{t_{i}},$$
  where $f$ is a function depending on $\omega_{D}$. By (\ref{equ Bound on the hold norm of the Ricci potential}), we have 
  \begin{equation}\label{equ F Ricci potential bounds}
  |F_{i}|_{\alpha^{\prime},\beta}\leq C.
  \end{equation}
 
  Without loss of generality, it suffices to show in  $B_{0}(R_{0})$, $R_{0}$ sufficiently small with respect to the background geometry (so a local coordinate system is defined),    $\frac{r_{p,\omega_{i}}}{d_{\beta,E}(p, \partial B_{0}(R_{0}))}$ is uniformly bounded away from $0$ independent of  $p$ and $i$.  

  We prove by contradiction. By Theorem \ref{Bootstrapping of Holder metrics} and Proposition \ref{Solvability of the ddbar equation}, if $r_{p,\omega_{i}}$ is not uniformly bounded away from $0$ independent of  $p$ and $i$, then   there exists a subsequence $(p_i,\omega_i),\ i\rightarrow \infty$ such that  
\begin{equation*} \frac{r_{p_{i},\omega_{i}}}{d_{\beta,E}(p_{i}, \partial B_{0}(R_{0}))}\rightarrow 0,\ p_{i}\rightarrow D,\   \textrm{and }
\end{equation*} 
\[0<\frac{r_{p_{i},\omega_{i}}}{d_{\beta,E}(p_{i},\partial B_{0}(R_{0}))}\leq 
2 \min_{p} \frac{r_{p,\omega_{i}}}{d_{\beta,E}(p, \partial B_{0}(R_{0}))}.\]
Next we consider the rescaled metric 
$\widehat{\omega}_i=r_{p_i,\omega_i}^{-2}T^{\star}_{r_{p_i,\omega_i}}\omega_i$ at $p_i$, where $T_{R}$ is defined as 
$$\widehat{z}\circ T_{R}=R^{\frac{1}{\beta}}z, \widehat{w}_{i}\circ T_{R}=Rw_{i}.$$

 The following properties  of  $\widehat{\omega}_i$ are obvious from the rescaling hypothesis and 
Proposition \ref{C^2 estimate} ($\widehat{\omega}_{\beta}$ and $\widehat{d}_{\beta,E}$ are the rescaled metric and distance in the rescaled coordinates).
\begin{itemize}
\item \begin{equation}\label{equ rescaled Monge ampere equation}\widehat{\omega}_i^n=e^{\widehat{F}_{i}}\widehat{\omega}_{\beta}^n. 
\end{equation} $\widehat{\omega}_i$ is defined on $B_{0}\{\frac{R_{0}}{r_{p_{i},\omega_{i}}}\}$,  $\widehat{F}_{i}$ is the pull back of $F_{i}$ via the rescaling map.
\item $\widehat{d}_{\beta,E}(p_{i}, \partial B_{0}(\frac{R_{0}}{r_{p_i,\omega_i}}))\rightarrow \infty$. 

\item For the same $K$ as in Proposition \ref{C^2 estimate}, we have 
\[\frac{1}{K} \widehat{\omega}_{\beta}\leq \widehat{\omega}_i\leq K\widehat{\omega}_{\beta}.\] 
\item By definition, for any $p\in  B_{0}(\frac{R_{0}}{r_{p_i,\omega_i}}))$ and $i$, we have $$r_{\widehat{\omega}_i,p}\geq \frac{\widehat{d}_{\beta,E}(p, \partial B_{0}(\frac{R_{0}}{r_{p_i,\omega_i}}))}{3\widehat{d}_{\beta,E}(p_{i}, \partial B_{0}(\frac{R_{0}}{r_{p_i,\omega_i}}))}.$$
Notice $\widehat{d}_{\beta,E}(p_{i}, \partial B_{0}(\frac{R_{0}}{r_{p_i,\omega_i}}))\rightarrow \infty$.
Consequently, suppose $\widehat{d}_{\beta,E}(p_{i}, p)<\lambda<\infty$ with respect to the rescaled Euclidean metric in polar coordinates, we have 
\begin{equation}\label{equ Holder radius bounded from below when i big for every pt}\liminf_{i\rightarrow \infty}r_{\widehat{\omega}_i,p}\geq \frac{1}{3}.
\end{equation} 
\item At $p_i$, we have $r_{\widehat{\omega}_i,p_i}= 1$.
\end{itemize}
\begin{clm}\label{Claim Bootstrapping in Elliptic estimate}For any $p$, when $i$ is large enough, the rescaled potential $\widehat{\phi}_i$ satisfies 
    \[|\widehat{\phi}_i|_{2,\alpha^{\prime},\beta,B_{p}(\frac{1}{100})}\leq C.\]
    \end{clm}
 To prove the claim,    without loss of generality we consider $p=0$. In $B_{0}(\frac{1}{2})$, by (\ref{equ Holder radius bounded from below when i big for every pt}),  when $i$ is large enough,   there exists a potential $\widehat\phi_{p,i}$ such that 
 \begin{itemize}
\item $\widehat{\omega}_{i}=\sqrt{-1}\partial \bar{\partial}\widehat\phi_{p,i}$,
\item $\widehat\phi_{p,i}\in C^{2,\alpha^{\prime},\beta}$, $[\widehat\phi_{p,i}]_{2,\alpha,\beta,B_{0}(\frac{1}{2})}\leq 4\delta_{0}$,\ $[\widehat\phi_{p,i}]_{2,\beta,B_{0}(\frac{1}{2})}\leq \underline{K}$,\\ $|\widehat\phi_{p,i}|_{0,B_{0}(\frac{1}{2})}\leq 8\widehat{K}$.
\end{itemize}
   Since $\delta_{0}$ is small enough in the sense of Definition \ref{Holder radius}, the proof of Proposition \ref{prop Gap} or the discussion of (37) in \cite{CDS2} directly imply the claim is true. 
   For the reader's convenience, we include the crucial step here. Without loss of generality, we assume $\widehat{\omega}_{i}$ satisfies the normalization condition at the point $0$: $\widehat{\omega}_{i}(0)=\omega_{\beta}$.   By the small ossilation condition ($[\widehat\phi_{p,i}]_{2,\alpha,B_{0}(\frac{1}{2}),\beta}\leq 4\delta_{0}$), we deduce  
\begin{equation}
[det(i\partial \bar{\partial}\widehat\phi_{p,i})-\Delta \widehat\phi_{p,i}]^{(\star)}_{\alpha^{\prime},B_{0}(\frac{1}{2})}
\leq \epsilon [i\partial \bar{\partial}\widehat\phi_{p,i}]_{\alpha^{\prime},B_{0}(\frac{1}{2})}^{(\star)},
\end{equation}
where $\epsilon$ is small enough with respect to $\delta_{0}$. 
Since $$det(i\partial \bar{\partial}\widehat\phi_{p,i})=e^{\widehat{F}_{i}}\in C^{\alpha^{\prime}},\ \textrm{in polar coordinates},  \alpha^{\prime}>\alpha,$$ combining (\ref{equ F Ricci potential bounds}), we deduce 
\begin{equation}\label{equ laplacian controlled by small constant times ddbar}
[\Delta \widehat\phi_{p,i}]^{(\star)}_{\alpha^{\prime},B_{0}(\frac{1}{2})}
\leq \epsilon [i\partial \bar{\partial}\widehat\phi_{p,i}]_{\alpha^{\prime},B_{0}(\frac{1}{2})}^{(\star)}+[e^{\widehat{F}_{i}}]_{\alpha^{\prime},B_{0}(\frac{1}{2})}\leq \epsilon [i\partial \bar{\partial}\widehat\phi_{p,i}]_{\alpha^{\prime},B_{0}(\frac{1}{2})}^{(\star)}+C.
\end{equation}
   Then  continuing as in the discussion after (37) in Chen-Donaldson-Sun's work \cite{CDS2}, or as (\ref{equ CDS small osc trick 1})--(\ref{equ bounding iddbar by c0 of phi})  in  the proof of Proposition \ref{prop Gap}, Claim \ref{Claim Bootstrapping in Elliptic estimate} is proved. 
  By (\ref{equ Bound on the hold norm of phi}) and (\ref{equ Bound on the hold norm of the Ricci potential}), the following crucial estimate is true. 
   \begin{equation}\label{Rescaled limit of the exponential term in the ckrf equation}
   \lim_{i\rightarrow \infty}e^{\widehat{F}_{i}}=C_{1}\ \textrm{uniformly on compact sets  over}\ \mathbb{C}^{n}\ \textrm{in}\
   C^{\alpha,\beta}-\textrm{topology},
   \end{equation}
   where $C_{1}$ is a positive constant. 
   
     Claim $\ref{Claim Bootstrapping in Elliptic estimate}$ implies  $\widehat{\omega}_{i}$ subconverge to  a $\omega_{\infty}$ over $\mathbb{C}^{n}$ locally  in $C^{\alpha,\beta}$-topology. Moreover, 
   \begin{itemize}
\item $\omega_{\infty}^n=C_{1}\omega_{\beta}^n$,
\item 
\begin{equation}\label{C1,1 bound for the rescaled limit}\frac{1}{K} \widehat{\omega}_{\beta}\leq \omega_{\infty}\leq K\widehat{\omega}_{\beta},\end{equation}
\item For any $p\in C^{n}$,   we have $r_{\omega_{\infty},p}\geq \frac{1}{3}$.
\end{itemize}

 We show in the following two cases, the above all lead to  contradictions.

 Case 1:  Suppose $d_{\beta,E}(p_{\infty},D)\leq  \frac{1000}{\beta\sin \beta\pi}$. By translation along the tangential direction of $D$, we can assume $d_{\beta,E}(p_{\infty},D)=d_{\beta,E}(p_{\infty},0)$. Under the translation along the tangential direction of $D$, the form of equation (\ref{equ rescaled Monge ampere equation}) is invariant, because $\omega_{\beta}$ is invariant under these tangential translations. This case  is the main issue (while the other cases are easier to handle).   From Theorem \ref{thm Liouville}, for some linear transformation $L$ which preserves $D= (0)\times C^{n-1}$, we have 
\[\omega_{\infty}=L^{\star}\omega_{\beta}.\]
By the proof of Proposition 25 in \cite{CDS2} and (\ref{C1,1 bound for the rescaled limit}), we obtain 
\begin{equation}\label{HE bound on a11 of the matrix L}
\frac{1}{K}\leq |a_{11}|^{2\beta}\leq K,\ \textrm{where }\ a_{11}\ \textrm{is the}\ (1,1)-\textrm{element of }\ L.
\end{equation}
Along the tangential direction of $D$, $L$ reduces to a $(n-1)\times (n-1)$ matrix $L_{T}$. By (\ref{C1,1 bound for the rescaled limit}) again, we get
\begin{equation}\label{HE bound on LT}
 |L_{T}|\leq CK^{\frac{1}{2}}.
\end{equation}

\begin{clm}\label{HE clm bound on the model potential}
Suppose $dist_{\beta}(p_{\infty},D)\leq \frac{1000}{\beta\sin \beta\pi}$. 
We can choose $\phi_{\infty}$ such that 
 $\omega_{\infty}=\sqrt{-1}\partial \bar{\partial}\phi_{\infty}$ and 
 \[|\phi_{\infty}|\leq CK\ \textrm{over}\ B_{p_{\infty}}(90).\]
\end{clm}

The proof of Claim \ref{HE clm bound on the model potential} is as follows.   Consider the most natural potential function $$\phi_{\infty}= L^{\star}(|z|^{2\beta}+\Sigma_{j=2}^{n}|w_{j}|^2).$$ 
 We obviously have   $\omega_{\infty}=L^{\star}\omega_{\beta}=\sqrt{-1}\partial \bar{\partial}\phi_{\infty}$.  By (\ref{HE bound on a11 of the matrix L}), (\ref{HE bound on LT}), and  the proof of  Proposition 25 in \cite{CDS2}, we directly obtain 
 \[|\phi_{\infty}|\leq CK\ \textrm{over}\ B_{p_{\infty}}(90).\]

   Obviously, we also have 
   \[[\phi_{\infty}]_{2,\beta}\leq CK\ \textrm{over}\ B_{p_{\infty}}(90),\ [\phi_{\infty}]_{2,\alpha,\beta}=0.\]
   
   The proof of Claim \ref{HE clm bound on the model potential} is completed.\\ 
   
      Now, when $i$ is sufficiently large,  we  perturb $\phi_{\infty}$ to be a potential $\underline{\phi}_i$ defined in $B_{p_i,\omega_i}(2)$ which satisfies the conditions in Definition \ref{Holder radius}, thus a contradiction will be obtained.  We  consider the equation 
 \[\sqrt{-1}\partial \bar{\partial}v_i=\omega_{i}-\omega_{\infty}.\]
 Notice  $|\omega_{i}-\omega_{\infty}|_{\alpha,\beta}\rightarrow 0$, uniformly over compact subdomains of $C^{n}$.  Using Proposition \ref{Solvability of the ddbar equation},
we obtain a solution $v_i$ such that 
\begin{equation}\label{equ bound on vi}|v_i|_{2,\alpha,\beta, B_{p_{i}}(2)}\leq C|\omega_{i}-\omega_{\infty}|_{\alpha,\beta, B_{p_{i}}(50)}.
\end{equation}
Thus the identity $\omega_i=\omega_{\infty}+\sqrt{-1}\partial \bar{\partial}v_i$ holds
in $B_{p}(2)$. By Proposition \ref{Solvability of the ddbar equation}
 and (\ref{equ bound on vi}), we have when $i$ is sufficient large that 
\begin{equation}\label{Bound on vi}    
[v_{i}]_{2,\alpha,\beta, B_{p_{i}}(2)}\leq \frac{\delta_{0}}{100}.
\end{equation}
Let $\underline{\phi}_{i}=\phi_{\infty}+v_i$. Notice the fact $[\omega_{\infty}]_{\alpha,\beta}=[L^{\star}\omega_{\beta}]_{\alpha,\beta}=0$ (as in Remark \ref{rmk omega beta has 0 ossilation in polar}) is quite important to show  the ossillation before rescaling is small. From  Claim \ref{HE clm bound on the model potential}, (\ref{equ Bound on the hold norm of the Ricci potential}), (\ref{equ Bound on the hold norm of phi}), (\ref{C1,1 bound for the rescaled limit}), and  (\ref{Bound on vi}),  by making  $\underline{K}$ and $\widehat{K}$ large enough, we obtain 
\begin{equation*}
[\underline{\phi}_{i}]_{2,\alpha,\beta,B_{p}(2)}\leq \frac{\delta_{0}}{100},\ [\underline{\phi}_{i}]_{2,\beta,B_{p}(2)}\leq \frac{\underline{K}}{2},\ |\underline{\phi}_{i}|_{0,B_{p}(2)}\leq \frac{\widehat{K}}{2}.
\end{equation*}
This is a contradiction since we assumed that there is no such potential for $\widehat{\omega}_i$ in a ball (centered at $p_i$) of radius larger than $1$!

Case 2:  Suppose $\infty> d_{\beta,E}(p_{\infty},D)> \frac{1000}{\beta\sin \beta\pi}$.  By translation along the tangential direction of $D$, we can also assume $d_{\beta,E}(p_{\infty},D)=d_{\beta,E}(p_{\infty},0)$. This case is easier, since before taking limit, the coordinate  $u=z^{\beta}$ is well defined in $B_{p_{i}}(90)$. This is because $B_{p_{i}}(90)$ does not cover a whole period $[0,2\pi]$ in this case, then we can choose the
   single-value branch of $z^{\beta}$ over  $[0,2\pi)$ in $B_{p_{i}}(90)$. Denote $p_{i}=(z^{\beta}_{i}, w_{1,i},...,w_{n-1,i})$.  Notice with respect to the coordinate $u=z^{\beta}, w_2,..., w_{n}$, we have \begin{equation}\label{equ omegabeta equals omegaE}\omega_{\beta}=\omega_{Euc},
   \end{equation}
   where $\omega_{Euc}$ is the Euclidean metric in the coordinates 
   $u, w_2,..., w_{n}$. 
   
   Hence, we still consider the origin $0$ as our base point. By exactly the small ossilation argument in case 1,  the  rescaled limit $\omega_{\infty}$ still equals $L^{\star}\omega_{\beta}$. Using (\ref{equ omegabeta equals omegaE}) and Proposition \ref{Solvability of the ddbar equation}, we  perturb the following potential 
   \[\underline{\phi}_{\infty}=L^{\star}(|z^{\beta}-z^{\beta}_{p_{\infty}}|^{2}+\Sigma_{j=2}^{n}|w_{j}-w_{j,p_{\infty}}|^2)\]
   to a potential before  $i$ goes to $\infty$, in $B_{p_{i}}(2)$ when $i$ is large enough.
   Then we get     the same  contradiction as in Case 1 to the hypothesis that there is no such potential in ball (centered at $p_{i}$) with radius larger than $1$ !

Case 3. Suppose $d_{\beta,E}(p_{i},D)\rightarrow \infty$. By translation along the tangential direction of $D$, we still  assume $d_{\beta,E}(p_{i},D)=d_{\beta,E}(p_{i},0)$. This  case is actually easier than Case 1 and  Case 2, because the almost smallest  H\"older radius occurs far away from $D$.  The argument is similar to Case 2. The difference is that, since in Case 3 the distance from $p_{i}$ to $D$ goes to $\infty$, we should choose $p_{i}$ as the base point of our convergence, not $0$ (as in case 1 and 2) anymore.  Still suppose  $p_{i}=(z_{i}, w_{2,i},...,w_{n,i})$, we denote the following coordinates as $\Psi_{i}$: $$\widehat{u}=z^{\beta}-z^{\beta}_{i}, \widehat{u}_{2}=w_{2}-w_{2,i},...,\widehat{u}_{n}=w_{n}-w_{n,i}.$$ 

With respect to the coordinate $\Psi_{i}$, We have \begin{equation}\label{equ omegabeta equals omegaE}\omega_{\beta}=\widehat{\omega}_{Euc},
   \end{equation}
   where $\widehat{\omega}_{Euc}$ is the Euclidean metric in the coordinates 
   $\widehat{u}, \widehat{u}_{2},..., \widehat{u}_{n}$. Then with respect to $\Psi_{i}$, by the translation invariance of $\widehat{\omega}_{Euc}$ along all directions (not only the tangential directions), the Monge-Ampere equation (\ref{equ rescaled Monge ampere equation}) is written as 
   \begin{equation}\label{equ translated rescaled Monge ampere equation}\widehat{\underline{\omega}}_i^n=e^{\widehat{\underline{F}}_{i}}\widehat{\omega}_{Euc}^n\ \textrm{in}\ B_{0}(\lambda_{i}),
\end{equation}
 where  $\widehat{\underline{F}}_{i}$ is the translated Ricci potential, and $$\lambda_{i}=\min\{\widehat{d}_{\beta,E}(p_{i},\partial \widehat{B}_{0}(\frac{R_{0}}{r_{p_{i},\omega_{i}}})),\ \frac{(\sin\beta\pi) \widehat{d}_{\beta,E}(p_{i},0)}{100}\} . $$
 Apparently, $\liminf_{i\rightarrow \infty}\lambda_{i}=+\infty$.
   Again by exactly the small ossilation argument in case 1, let $i\rightarrow \infty$, $\widehat{\underline{\omega}}_i$ tends to $\widehat{\underline{\omega}}_{\infty}$ strongly in the $C^{\alpha}$-sense, over compact subdomains of $C^{n}$. The limit $\widehat{\underline{\omega}}_{\infty}$  satisfies 
     \begin{equation}\label{equ limit translated rescaled Monge ampere equation}\widehat{\underline{\omega}}_{\infty}^n=C_{2}\widehat{\omega}_{Euc}^n\ \textrm{in}\ \mathbb{C}^n,\ \frac{\widehat{\omega}_{Euc}}{K}\leq \widehat{\underline{\omega}}_{\infty}\leq K\widehat{\omega}_{Euc}.
\end{equation}

    By Theorem \ref{thm Liouville} (in the case when $\beta=1$),  we still  have  $$\omega_{\infty}=L^{\star}\omega_{\beta}\ \textrm{over}\ C^{n}.$$ Using (\ref{equ omegabeta equals omegaE}) and Proposition \ref{Solvability of the ddbar equation}, we  perturb the following potential of $\widehat{\underline{\omega}}_{\infty}$
   \[\widehat{\underline{\phi}}_{\infty}=L^{\star}(|\widehat{u}|^{2}+\Sigma_{j=2}^{n}|\widehat{u}_{j}|^2)\ \textrm{in terms of the coordinate}\ \Psi_{i}. \]
   to a potential before  $i$ goes to $\infty$, in $B_{0}(2)$. Then, we obtain a contradiction as in Case 1 and Case 2 again,  to the hypothesis that $\frac{r_{p_{i},\omega_{i}}}{d_{\beta,E}(p_{i}, \partial B_{0}(R_{0}))}$ goes to 0!

   Thus, $\frac{r_{p_{i},\omega_{i}}}{d_{\beta,E}(p_{i}, \partial B_{0}(R_{0}))}$ can not go to $0$. This shows \begin{equation}\label{equ bound on phi c2alpha near D}\phi]_{2,\alpha,\beta,T_{\frac{R_{0}}{2}}(D)}\leq C, \end{equation}
   where $T_{\frac{R_{0}}{2}}(D)$ is the tubular neighborhood of $D$ with width $\frac{R_{0}}{2}$. By parabolic Evans-Krylov-Safanov Theorem (as in \cite{Wanglihe1}), we deduce the following estimate away from  $D$
  \begin{equation}\label{equ bound on phi c2alpha away from D}[\phi]_{2,\alpha,\beta,M\setminus T_{\frac{R_{0}}{4}}(D)}\leq C. 
  \end{equation}
   (\ref{equ bound on phi c2alpha near D}) and (\ref{equ bound on phi c2alpha away from D}) imply 
   \[[\phi]_{2,\alpha,\beta,M}\leq \mathbb{K}.\]
    The proof of Theorem \ref{Evans Krylov} is complete.
 
Step 3:\ 
To prove the long time existence part, notice that by the proof of Theorem 1.2 in \cite{CYW}, the short time $t_0$ such that the CKRF exists only depend on the background geometry $(M,(1-\beta)D,\omega_0)$ and $|\phi_0|_{2,\alpha,\beta}$, where $\phi_0$ is the potential of the initial metric   with respect to the reference metric $\omega_D$. Since $|\phi(t)|_{2,\alpha,\beta}\leq \mathbb{K}$ which is independent of $t\in [0,T)$, we can start the short-time solution for time period $t_0$ from 
$\phi(T-\frac{t_0}{2})$, thus end up with a flow for  $t\in [0,T+\frac{t_0}{2}]$. The $t_{0}$ is the short existence time in Theorem 1.2
of \cite{CYW}, subject to the bound $\mathbb{K}$ and the background geometry. Then the flow can be extended beyond any finite $T>0$.

The proof of the long time existence is completed.

Since $T\geq t_{0}$, where $t_{0}$ is the 
short existence time of the CKRF in Theorem 1.2 of \cite{CYW},  from the proof in Step 1, we conclude that $\mathbb{K}$ depends on the background geometry $(M,L,h,\omega_{0})$, the $C^{1,1}$-bound on $\phi$,  $|\frac{\partial \phi}{\partial t}|_{0}$, and the initial  metric of the flow. 

\end{proof}

\begin{rmk} We actually proved more: when the volume form (with respect to $\omega_{\beta}$) is $C^{\alpha^{\prime},\beta}$, we can obtain $C^{\alpha,\beta}$ estimate on the second derivatives ($\alpha<\alpha^{\prime}$), provided the $C^{1,1}$-estimate is already obtained.
This is interesting even in smooth case (when $\beta$=1), and we will discuss it in detail in a sequel of this paper.
\end{rmk}

\section{Poincare-Lelong equations.\label{section Poincare Lelong equations}}
In this section we work in the holomorphic coordinates.
 Our main target is to prove Proposition \ref{Solvability of the ddbar equation}. 
This  is crucial in the proof of Theorem \ref{Evans Krylov},  when we perturb the potential of the rescaled limit metric back to a potential before taking limit to get a contradiction (as in  \cite{Anderson} , where the Laplace equation is the main interest). Let $A_{R}$ be the cylinder (centered at $0$) with respect to the model cone metric $\omega_{\beta}$, as in \cite{CYW}. Let $\omega_{E}$ be the Euclidean metric in the holomorphic coordinates.
 
\begin{prop}\label{Solvability of the ddbar equation}There exists a constant $C$ depending on $\beta$ and $n$ with the following properties. Given the equation 
\begin{equation}\label{ddbar equation}\sqrt{-1}\partial \bar{\partial}v=\eta\ \textrm{over}\ A_{20},
\end{equation}
where $\eta\in C_{1,1}^{\alpha,\beta}$ is a closed (1,1)-form such that $\eta=\sqrt{-1}\partial\bar{\partial}\phi_{\eta}$ for some $\phi_{\eta}\in C^{2,\alpha,\beta}$. Then there exists a solution $v$ in 
$C^{2,\alpha,\beta}$ such that 
\begin{enumerate} \item $|v|_{2,\alpha,\beta,A_{5}}\leq C|\eta|_{\alpha,\beta,A_{20}}.$
\item $|v|_{0,A_{5}}\leq C|\eta|_{0,\beta,A_{20}}.$
\end{enumerate}

\end{prop}
\begin{rmk}By the assumptions, (\ref{ddbar equation})  is already solved
by $\phi_{\eta}$. The point is that we want a solution with the correct estimate.
\end{rmk}
\begin{proof}{ of Proposition \ref{Solvability of the ddbar equation}:} We only need  to find a  solution $v\in W^{1,2}_{\omega_{\beta}}(A_{6})\cap C^{0}(A_{6})$ such that  
\[|v|_{W^{1,2}_{\omega_{\beta}}(A_{6})}\leq C|\eta|_{0,\beta,A_{20}},\]
so consequently $v$ is  a weak solution to (\ref{ddbar equation}), by Lemma 2.5 in \cite{WYQ}. Then the Schauder regularity  estimate in \cite{Don} or in \cite{CYW} implies  $v$ is in $C^{2,\alpha,\beta}$ and $v$ satisfies interior Schauder estimate. 

    With the help of   Lemma \ref{Weak solution to the ddbar- equation},
 the    $W^{1,2}_{\omega_{\beta}}$-estimate of $v$ is actually straightforward.  It suffices to observe that 
\begin{equation}
\int_{A_{10}}\beta^2|z|^{2-2\beta}|\frac{\partial v}{\partial z}|^2 \omega^n_{\beta}=\int_{A_{10}} |\frac{\partial v}{\partial z}|^2 \omega^n_{Euc}\leq C.
\end{equation}
Obviously we have  $\Sigma_{i=1}^{n-1}|\frac{\partial v}{\partial u}|_{0,A_{10}}\leq C|\eta|_{0,\beta,A_{20}}$, then by Lemma 2.5 in \cite{WYQ}, $v$ is actually a weak solution to the following trace equation 
\begin{equation}\label{ddbar equation}\Delta_{\beta}v=\eta\ \textrm{over}\ A_{10}.
\end{equation}  
Thus the Moser's iteration trick  works again, as in the proof of Lemma  \ref{lem Trudinger's Harnack inequality}. Thus $v\in C^{0,\beta}(A_{10})\cap W^{1,2}_{\omega_\beta}$ and 
$$|v|_{0,A_{5}}\leq C|\eta|_{0,\beta,A_{10}}.$$
Item 2 is thus proved. 

 By 
the main Theorem in \cite{Don} and item 1 in Lemma \ref{Weak solution to the ddbar- equation} , we conclude   $v\in C^{2,\alpha,\beta}$ and 
\begin{equation}
|v|_{2,\alpha,\beta,A_{5}}\leq C(|\eta|_{\alpha,\beta, A_{10}}+|v|_{0,A_{10}})\leq C|\eta|_{\alpha,\beta, A_{20}}.
\end{equation}
The proof of item 1 is also complete.
\end{proof}
 Consider the natural orbifold map 
 \begin{eqnarray*}
 \mathfrak{T}:& & A_{\frac{1}{N}, 20^{\frac{1}{\beta N}},20}\rightarrow A_{\frac{1}{N},20^{\frac{1}{\beta N}},20}
    \\& &\mathfrak{T}(w)=w^{N}=z,
 \end{eqnarray*}
 where $A_{\frac{1}{N},R_{1},R_{2}}$ means the cylinder (centered at $0$) of normal radius $R_{1}$ and tangential radius $R_{2}$, with respect to the orbifold model metric $\omega_{\frac{1}{N}}$

\begin{lem}\label{C2alphabeta space gets more singular when beta get smaller}Suppose 
$\beta>\beta_{0},\ \alpha<1$, then 
\[C^{\alpha,\beta}_{1,1} \in C^{\widehat{\alpha},\beta_{0}}_{1,1},\]
where $\widehat{\alpha}=\min\{\alpha, \frac{\beta}{\beta_0}-1\}$ and
$C^{\alpha,\beta}_{1,1}$ is the space of $C^{\alpha,\beta}$ (1,1)-forms. 
Moreover, suppose $\beta>\frac{1}{N-1}>\frac{1}{N}$, and  $\eta=\sqrt{-1}\partial\bar{\partial}\phi_{\eta}$ for some $\phi_{\eta}\in C^{2,\alpha,\beta}$, then we can pull back $\eta$ by $\mathfrak{T}$  such that  $\mathfrak{T}^{\star}\eta \in C^{\acute{\alpha}}\ \textrm{in the usual sense upstairs},$
where $\acute{\alpha}<\min(\frac{1}{2(N-1)},\alpha)$.
\end{lem}
\begin{proof}{of Lemma \ref{C2alphabeta space gets more singular when beta get smaller}: }

This Lemma is easy to prove as follows. With respect to $(1,1)$-derivative, $\eta \in C^{\alpha,\beta}_{1,1}$  means 
\[|z|^{2-2\beta}\eta_{(z\bar{z})}\in C^{\alpha,\beta}.\]
 Thus we have 
 \[|z|^{2-2\beta_0}\eta_{(z\bar{z})}=[|z|^{2-2\beta}\eta_{(z\bar{z})}] (|z|^{2\beta-2\beta_0}). \]
 Since $[|z|^{2-2\beta}\eta_{(z\bar{z}}]\in C^{\alpha,\beta}\in C^{\alpha,\beta_0}$
 , and $(|z|^{2\beta-2\beta_0}\in C^{\min(\frac{2\beta}{\beta_0}-2,1),\beta_0}$, then 
 \[|z|^{2-2\beta_0}\eta_{(z\bar{z})}\in C^{\min(\frac{2\beta}{\beta_0}-2,\alpha),\beta_0}.\]
 Notice for the mixed derivatives we have  for any $1\leq i\leq n-1$ that 
 \begin{equation}
 |z|^{1-\beta_0}\eta_{\rho\bar{u}_i}=|z|^{\beta-\beta_0} |z|^{1-\beta}\eta_{\rho,u_i}.
 \end{equation}
 
 Using the assumption \[ |z|^{1-\beta}\eta_{\rho,u_i}\in C^{\alpha,\beta} \in C^{\alpha,\beta_0}.\]
  and the fact $|z|^{\beta-\beta_0}\in C^{\min (\frac{\beta}{\beta_0}-1,\alpha);\beta_0}$
 we get 
  \[|z|^{1-\beta_0}\eta_{z\bar{u}_i}\in C^{\min (\frac{\beta}{\beta_0}-1,1),\beta_0}.\]
  
  Usually we can not pull back a current $\eta$ . However, in case 
  when $\eta=\sqrt{-1}\partial\bar{\partial}\phi_{\eta}$ for some $\phi_{\eta}\in C^{2,\alpha,\beta}$, we can pull back $\eta$ by defining \[\mathfrak{T}^{\star}\eta= \sqrt{-1}\partial\bar{\partial}\mathfrak{T}^{\star}\phi_{\eta}.\] 
  Then, for the last part in Lemma \ref{C2alphabeta space gets more singular when beta get smaller}, without of generality we only consider 
  the mixed term $(\mathfrak{T}^{\star}{\eta})_{w\bar{u}_i}$, the other terms are similar. Notice that 
  \begin{eqnarray}\label{Mixed derivative Holder in PL equation}
 \nonumber (\mathfrak{T}^{\star}{\eta})_{w\bar{u}_i}& =&Nw^{N-1}\mathfrak{T}^{\star}(|z|^{\beta-1}|z|^{1-\beta} \eta_{z\bar{u}_i})
 \nonumber \\&=&Nw^{N-1}|w|^{N\beta-N}e^{-iN\theta_{w}}\mathfrak{T}^{\star}(|z|^{1-\beta} \eta_{\frac{\partial}{\partial \rho}+\frac{1}{i\rho}\frac{\partial}{\partial \theta_{z}},\bar{u}_i})
  \nonumber\\&=&N\frac{w^{N-1}}{|w|^{N-1}}|w|^{N\beta-1}e^{-iN\theta_{w}}\mathfrak{T}^{\star}(|z|^{1-\beta} \eta_{\frac{\partial}{\partial \rho}+\frac{1}{i\rho}\frac{\partial}{\partial \theta_{z}},\bar{u}_i})
 \nonumber \\&=&N|w|^{N\beta-1}e^{-i\theta_{w}}\mathfrak{T}^{\star}(|z|^{1-\beta} \eta_{\frac{\partial}{\partial \rho}+\frac{1}{i\rho}\frac{\partial}{\partial \theta_{z}},\bar{u}_i}).
  \end{eqnarray}
  Since  \[|z|^{1-\beta} \eta_{\frac{\partial}{\partial \rho}+\frac{1}{i\rho}\frac{\partial}{\partial \theta_{z}},\bar{u}_i}\in C^{\alpha,\beta}\in C^{\alpha,\frac{1}{N}},\]
   then $\mathfrak{T}^{\star}(|z|^{1-\beta} \eta_{\frac{\partial}{\partial \rho}+\frac{1}{i\rho}\frac{\partial}{\partial \theta_{z}},\bar{u}_i})\in C^{\alpha}$ in the regular sense upstairs. On the other hand, by (\ref{lower bound for the angle beta}) we have $|w|^{N\beta-1}e^{-i\theta_{w}}\in C^{\frac{1}{2(N-1)}}$. 
   
   The proof of the Lemma is completed. \end{proof}
\begin{lem}\label{Weak solution to the ddbar- equation}Under the same hypothesis of Proposition \ref{Solvability of the ddbar equation}.  There exists  a weak solution $v\in W^{1,2}_{\omega_{E}}$  to equation (\ref{ddbar equation}) such that 
\begin{itemize}
\item $|v|_{0,A_{10}}\leq C|\eta|_{\alpha,\beta, A_{20}}$,
\item $|\frac{\partial v}{\partial z}|_{L^2(A_{10}),\omega_{E}}+\Sigma_{i=1}^{n-1}|\frac{\partial v}{\partial u}|_{0,A_{10}}\leq C|\eta|_{\alpha,\beta, A_{20}}$.
\end{itemize}
\begin{proof}{of Lemma  \ref{Weak solution to the ddbar- equation}:}We use the orbifold method, which should be counted as a geometric argument. Fix a $1>\beta>0$, there exists an integer $N$ such that 
\begin{equation}\label{lower bound for the angle beta}
 \beta>\frac{1}{N-1}>\frac{1}{N}.
\end{equation}

 The geometry of $A_{20}$ is like an orbifold. Moreover,  it's obvious that 
 
 \[A_{20}=A_{\frac{1}{N},20^{\frac{1}{\beta N}},20}, \]
 where $A_{\frac{1}{N},20^{\frac{1}{\beta N}},20}$ is the cylinder (centered at $0$) of normal radius $20^{\frac{1}{\beta N}}$ and tangential radius $20$, with respect to the orbifold model metric $\omega_{\frac{1}{N}}$.

 Therefore  we could treat $\eta$ as a form in $C^{\widehat{\alpha},\frac{1}{N}}_{1,1}$, $\widehat{\alpha}<\min(\frac{1}{2(N-1)},\alpha)$.
 
 The we consider the $i\partial\bar{\partial}$-equation over the upstair space
 \begin{equation}\label{unnormalized ddbar equation upstairs }
 i\partial\bar{\partial}\widehat{v}= \mathfrak{T}^{(\star)}\eta.
 \end{equation}
By Claim \ref{C2alphabeta space gets more singular when beta get smaller} we obtain  \[|\mathfrak{T}^{(\star)}\eta|_{\acute{\alpha},A_{\frac{1}{N}, 20^{\frac{1}{\beta N}},20}}\leq C|\eta|_{\alpha, \beta,A_{20}}.\] Using Hormander's results in \cite{Hormander} and the standard proof  of the $i\partial\bar{\partial}$-lemma as in \cite{GH}, we can find a solution $\widehat{v}$ to equation (\ref{unnormalized ddbar equation upstairs })  with the following properties.
 \begin{eqnarray}\label{PL: estimate upstairs}
 & & |\widehat{v}|_{0,A_{\frac{1}{N}, 10^{\frac{1}{\beta N}},10}}+|\frac{\partial \widehat{v}}{\partial w}|_{0,A_{\frac{1}{N}, 10^{\frac{1}{\beta N}},10}}+\Sigma_{i=1}^{n-1}|\frac{\partial \widehat{v}}{\partial u}|_{0,A_{\frac{1}{N}, 10^{\frac{1}{\beta N}},10}}
\nonumber \\&\leq &  C|\mathfrak{T}^{(\star)}\eta|_{0,A_{\frac{1}{N}, 20^{\frac{1}{\beta N}},20}}.
\end{eqnarray}  
Denote $a_{N}=e^{\frac{2\pi i}{N}}$ as the $n-th$ unit root,  we define the renormalized solution as 
\begin{equation} \label{invariant solution upstairs}
\underline{v}(w,\ \cdot)=\frac{1}{N}[\widehat{v}(w,\ \cdot)+\widehat{v}(a_{N}w,\ \cdot)+...+\widehat{v}(a^{N-1}_{N}w,\ \cdot)].
\end{equation}
Then $\underline{v}$ is invariant under the deck transformation over $A_{\frac{1}{N}, 20^{\frac{1}{\beta N}},20}$ (by multiplying $a_{N}$). Moreover, $\underline{v}$ still solves (\ref{unnormalized ddbar equation upstairs }). By (\ref{PL: estimate upstairs}),  $\underline{v}$ satisfies 
 \begin{eqnarray}\label{PL:Estimate upstairs on the symmetrized solution}
 & & |\underline{v}|_{0,A_{\frac{1}{N}, 10^{\frac{1}{\beta N}},10}}+|\frac{\partial \underline{v}}{\partial w}|_{0,A_{\frac{1}{N}, 10^{\frac{1}{\beta N}},10}}+\Sigma_{i=1}^{n-1}|\frac{\partial \underline{v}}{\partial u}|_{0,A_{\frac{1}{N}, 10^{\frac{1}{\beta N}},10}}\nonumber
 \\&\leq& C|\mathfrak{T}^{(\star)}\eta|_{0,A_{\frac{1}{N}, 20^{\frac{1}{\beta N}},20}}.
\end{eqnarray}  
Now we show that $\underline{v}(w)$ can be decended to $v(z)$ over $A_{10}$. Suppose $z=\rho e^{i\theta}, \theta\in [0,2\pi)$.  Define
 \begin{equation*}
 v(z)\triangleq \underline{v}(z^{\frac{1}{N}}),
\end{equation*}  
where $z^{\frac{1}{N}}$ represents the single-valued branch as $re^{i\theta}\rightarrow r^{\frac{1}{N}}e^{\frac{i\theta}{N}}$. 

It's easy to check $ v(z)$ has the same limit when $\theta$ approaches $0$ and $2\pi$, therefore $v(z)$ is well defined. This is obvious from the construction in  (\ref{invariant solution upstairs}). It's also easy to check $\frac{\partial v}{\partial z}(\frac{\partial v}{\partial \bar{z}})$, $\frac{\partial^2 v}{\partial z\partial \bar{z}}$, $\frac{\partial^2 v}{\partial z^2}(\frac{\partial^2 v}{\partial \bar{z}^2})$  all match up when $\theta=0$ and $2\pi$, therefore they  are well defined and at least continous away from $\{z=0\}$.\\

 By construction we directly have  $v\in C^{2,0,\frac{1}{N}}$. Moreover,   the $C^{0}$ estimate upstairs trivially decends downstairs, namely we have
    \begin{equation}\label{PL:C^0 estimate downstairs}
    |v|_{0,A_{10} }\leq C|\eta|_{0,\frac{1}{N},A_{\frac{1}{N}, 20^{\frac{1}{\beta N}},20}}\leq C|\eta|_{\alpha,\beta,A_{20}}.
     \end{equation}
    $v(z)$ is $C^{2,\alpha}$ in the usual sense away from $D$.
    
    Next we consider $W^{1,2}_{\omega_{E}}$ estimates in the holomorphic coordinates
    (with respect to the Euclidean metric). 
The estimate upstairs $$|\frac{\partial \underline{v}}{\partial \omega}|_{0,A_{\frac{1}{N}, 10^{\frac{1}{\beta N}},10}}\leq C|\mathfrak{\eta}|_{0,\frac{1}{N},A_{\frac{1}{N}, 20^{\frac{1}{\beta N}},20}}$$
 and the second inquality in (\ref{PL:C^0 estimate downstairs}) implies 
 \begin{equation}
 N^2|z|^{2-\frac{2}{N}}|\frac{\partial v}{\partial z}|^2\leq C|\eta|^2_{0,\beta,A_{20}}\ \textrm{ for all }\ z\in A_{10}.
 \end{equation}
 Then we have 
\begin{equation*}
\int_{A_{10}} |\frac{\partial v}{\partial z}|^2 \omega^n_{E}=\int_{A_{10}} N^2|z|^{2-\frac{2}{N}}|\frac{\partial v}{\partial z}|^2 \omega^n_{\frac{1}{N}}\leq C|\eta|^2_{\alpha,\beta,A_{20}}.
\end{equation*}
The tangential derivatives are obviously bounded in $C^{0}$-norm. Thus the proof of 
Lemma \ref{Weak solution to the ddbar- equation} is complete.
\end{proof}
\end{lem}

\section{A rigidity theorem.\label{section A rigidity theorem}}
In this section we prove Theorem \ref{thm Tangent cone trivial implies omega trivial}. This theorem   implies Theorem \ref{thm Liouville}, if we can show $\omega$ has a tangent cone  which is isomorphic to $\omega_{\beta}$.

\begin{thm}\label{thm Tangent cone trivial implies omega trivial}Suppose $\omega$ is a $ C^{\alpha,\beta}$ conical K\"ahler metric defined  over $C^{n}$. Suppose 
\begin{equation}
\omega^{n}=\omega_{\beta}^{n},\ \frac{\omega_{\beta}}{K}\leq \omega\leq K\omega_{\beta}\ \textrm{over}\ C^{n}.
\end{equation}
Suppose for some linear transformation $L$, $L^{\star}\omega_{\beta}$ is one of the tangent cones of $\omega$. Then $\omega=L^{\star}\omega_{\beta}$. 
\end{thm}
\begin{proof}{of Theorem \ref{thm Tangent cone trivial implies omega trivial}:} Without loss of generality we assume $L=id$. 
   Consider   scalling  $B(R)$ to $B(1)$ as 
\begin{equation}\label{equ scale down formula}
\widehat{z}\rightarrow R_{i}^{-\frac{1}{\beta}}\widehat{z}=z;\
\widehat{u}_i\rightarrow R_{i}^{-1}\widehat{u}_{i}=u_i;\
\widehat{\phi}\rightarrow R_{i}^{-2}\widehat{\phi}=\phi.
\end{equation}
Suppose the tangent cone along the sequence $R_{i}$ is $\omega_{\beta}$ , which means $\omega_{i}=R^{-2}_{i}\omega\rightarrow \omega_{\beta}$ over $B(\lambda)$ for all $\lambda>0$. Take $\lambda=1$, we have from the proof of Proposition 2.5 in \cite{CDS1} that 
      \[\lim_{i}|\omega_{i}-\omega_{\beta}|_{L^{2}(B_{0}(1))}=0.\]
  
  By the Moser's iteration trick in the proof of Proposition 26 in \cite{CDS2}, since $\omega_{i}$ is also  Ricci flat and quasi-isometric 
  to $\omega_{\beta}$ in the scaled down coordinates, we have 
   \[\lim_{i}|\omega_{i}-\omega_{\beta}|_{L^{\infty}, (B_{0}(\frac{1}{2}))}=0.\]
   
   Thus, when $i$ is large enough, $\omega_{i}$ satisfies the assumptions in Proposition \ref{prop Gap} over $B(\frac{1}{2})$. Then we obtain when $i$ is large that 
        \[[\omega_{i}]_{\alpha,\beta,B(\frac{1}{8})}\leq C.\]
       Rescale back, we get 
       \[[\omega]_{\alpha,\beta,B(\frac{R_{i}}{8})}\leq CR_{i}^{-\alpha}.\]
       
       Let $i\rightarrow \infty$, we get 
        \[[\omega]_{\alpha,\beta,\ C^{n}}=0.\]
        Then $\omega=\omega_{\beta}$ over $C^{n}$. The proof is complete.



\end{proof}

\begin{prop}\label{prop Gap}Suppose $\omega$ is a $C^{\alpha,\beta}$ conical K\"ahler metric  defined  over $B_{0}(1)$. Suppose there is a small enough $\delta$ such that  
\begin{equation}
\omega^{n}=\omega_{\beta}^{n},\  \frac{\omega_{\beta}}{1+\delta}\leq \omega\leq (1+\delta)\omega_{\beta}\ \textrm{over}\ B_{0}(1).
\end{equation}
Then the following estimate holds in $B(\frac{1}{4})$. 
\[[\omega]_{\alpha,\beta,B(\frac{1}{4})}\leq C.\]
\end{prop}
\begin{proof}{of Proposition \ref{prop Gap}:}\ By the solution to the Poincare-Lelong equation, we obtain a potential $\phi$ such that 
\begin{equation}\label{equ bound on C0 norm of phi by poincare lelong solution}
i\partial \bar{\partial}\phi=\omega,\ |\phi|_{0,B(\frac{4}{5})}\leq C.
\end{equation}

Under the singular coordinates and the basis $\mathfrak{a},\ du_{1},...\ du_{n-1}$ for $T^{1,0}$, we consider $i\partial \bar{\partial} \phi$ under these basis, as in page 11 of \cite{CYW}. 

 Then note that, by letting $F(M)=detM-trM$, we consider 
 \[|F(i\partial \bar{\partial} \phi(x))-F(i\partial \bar{\partial} \phi(y))|.\]
Since $(1-\delta)I\leq i\partial \bar{\partial}\phi\leq (1+\delta)I$, we obtain
\begin{equation}
|F(i\partial \bar{\partial} \phi(x))-F(i\partial \bar{\partial} \phi(y))|\leq \epsilon |i\partial \bar{\partial} \phi(x)-i\partial \bar{\partial} \phi(y)|,
\end{equation}
for some $\epsilon(\delta)$ such that $\lim_{\delta\rightarrow 0}\epsilon(\delta)=0$. 
Hence 
\begin{equation}
[det(i\partial \bar{\partial}\phi)-\Delta \phi]^{(\star)}_{\alpha,B(1)}
\leq \epsilon [i\partial \bar{\partial}\phi]_{\alpha,B(1)}^{(\star)}.
\end{equation}
Since $det(i\partial \bar{\partial}\phi)=1$ in polar coordinates, we deduce 
\begin{equation}\label{equ laplacian controlled by small constant times ddbar}
[\Delta \phi]^{(\star)}_{\alpha,B(1)}
\leq \epsilon [i\partial \bar{\partial}\phi]_{\alpha,B(1)}^{(\star)}.
\end{equation}
Combining (\ref{equ laplacian controlled by small constant times ddbar}) and  the usual conic Schauder estimate
\begin{equation}\label{equ CDS small osc trick 1}
[i\partial \bar{\partial}\phi]_{\alpha,\beta,B(1)}^{(\star)}\leq C\{[\Delta \phi]^{(\star)}_{\alpha,\beta,B(1)}+|\phi|_{0,B(1)}\},
\end{equation}
we end up with 
\begin{equation}
[i\partial \bar{\partial}\phi]_{\alpha,\beta,B(1)}^{(\star)}\leq C\epsilon[i\partial \bar{\partial} \phi]^{(\star)}_{\alpha,\beta,B(1)}+C|\phi|_{0,B(1)}.
\end{equation}
Let $\delta$ be small enough such that $C\epsilon<\frac{1}{2}$, we deduce
\begin{equation}\label{equ bounding iddbar by c0 of phi}
[i\partial \bar{\partial}\phi]_{\alpha,\beta,B(1)}^{(\star)}\leq C|\phi|_{0,B(1)}\leq C.
\end{equation}

The proof is complete. 
\end{proof}

 \section{ Bounded weakly-subharmonic functions and weak maximum principle. \label{section Bounded weakly-subharmonic functions and weak maximum principle}}
   In this section, we work in the polar coordinates (the balls, domains are all with respect to the polar coordinates). We mainly show the Dirichlet boundary problem 
   is solvable, in the sense of Theorem \ref{thm Dirichlet problem for C2alpha boundary value} and \ref{thm Dirichlet problem for continuous boundary value}. These are important in the last part of the proof of Theorem \ref{thm Liouville} in section \ref{section Proof of Theorem liouville}.

 Following  \cite{CDS2}, the following Lemma is true on bounded weakly-harmonic and  weakly-subharmonic functions. 
 \begin{lem}\label{lem Linfty solution is actually a weak solution}

Suppose  $u\in C^{2}(B(1)\setminus D)\cap L^{\infty}(B(1))$. Then 
\begin{enumerate}\item Suppose $\Delta_{\omega}u\geq 0$ over $B(1)\setminus D$, then $u$ is a weak subsolution to $\Delta_{\omega}u\geq 0$ in $B(1)$;
\item Suppose $\Delta_{\omega}u= 0$ over $B(1)\setminus D$, then $u$ is a weak solution to $\Delta_{\omega}u= 0$ in $B(1)$,  and $u\in C^{\alpha,\beta}$ for some $\alpha>0$.
\end{enumerate} 
\end{lem}
\begin{proof}This is proved by cutting off. Let $\eta_{\epsilon}=\Psi(\frac{1}{3}-d_{p}) \Psi(\frac{r}{\epsilon}-1) $, where $\Psi$ is  the Lipshitz cutoff function 
\begin{displaymath}\Psi(s)=\left \{
\begin{array}{ccr}& 0,\ s\leq 0;\\
 & 6s,\ 0\leq s\leq \frac{1}{6};\\
   & 1,\  \frac{1}{6}\leq s\leq \frac{1}{3}.\\
\end{array} \right.
\end{displaymath}
 Notice that $\Psi^{\prime}(s)\leq 6$ almost everywhere, then since $\omega$ is quasi isometric to the Euclidean metric in the polar coordinates, we have when $\epsilon \leq \frac{1}{100}$ that 
 \begin{equation}\label{equ gradient bound of etaepsilon}
 |\nabla_{\omega}\eta_{\epsilon}|\leq \frac{C}{\epsilon}.
 \end{equation}
$\eta_{\epsilon}$ not only cutoff the boundary of $B(1)$, but also cutoff the divisor $D$. Since $u$ is smooth away from divisor, we multiply both handsides of the harmonic equation (in item 2) by $\eta_{\epsilon}^2 u$ and integrate by parts to get  

\begin{equation}
-\int\eta_{\epsilon}^2|\nabla_{\omega} u|^2-2\int \eta_{\epsilon} u<\nabla_{\omega} u,\ \nabla_{\omega} \eta_{\epsilon}>=0.
\end{equation}
Then by Cauchy-Schawartz inequality, we get
\begin{equation}\label{equ bound on gradient of etaepsilon times u 1}
\frac{1}{2}\int\eta_{\epsilon}^2|\nabla_{\omega} u|^2\leq C\int |\nabla_{\omega}\eta_{\epsilon}|^2 u^2.
\end{equation}
Thus, by the condition $|u|_{L^{\infty}}< \infty$, the bound (\ref{equ gradient bound of etaepsilon}), and the definition of $\eta_{\epsilon}$, we obtain 
\begin{equation}\label{equ bound on gradient of etaepsilon times u 2}
\int |\nabla_{\omega}\eta_{\epsilon}|^2 u^2\leq C|u|_{L^{\infty}}^2\int_{\epsilon}^{\frac{7\epsilon}{6}}\frac{1}{\epsilon^2}rdr\leq  C|u|_{L^{\infty}}^2. 
\end{equation} 

Hence (\ref{equ bound on gradient of etaepsilon times u 1}) and (\ref{equ bound on gradient of etaepsilon times u 2}) imply 
\begin{equation}
\int\eta_{\epsilon}^2|\nabla_{\omega} u|^2\leq C,
\end{equation}
where $C$ is independent of epsilon! Therefore let $\epsilon\rightarrow 0$, we get 
\begin{equation}\int_{B(\frac{1}{6})\setminus D}|\nabla_{\omega} u|^2\leq C.
 \end{equation}
 By Lemma 2.5 of \cite{WYQ}, $u$ is a weak solution to (\ref{equ omega harmonic equation}). By Theorem 8.22 of \cite{GT} or \cite{LSU}, 
 we deduce $u\in C^{\alpha,\beta}$.
 
  The statement on subharmonicity is proved in the same way, by considering $u-\inf{u}_{B(1)}$, which is nonnegative. 
\end{proof}

 Recall the classical \emph{weak maximum principle} for the subharmonic function on Euclidean space. Suppose $\Omega \subset \mathbb{R}^n$ is a bounded open subset, if $u\in C^2(\Omega)\cap C^0(\bar{\Omega})$ satisfies $\Delta u\geq 0$, then $\sup_\Omega u=\sup_{\partial\Omega} u$.

Let $\Omega\subset \mathbb{C}^n$ be a connected bounded open subset which intersects $D$, let $\omega$ be a weak conical K\"ahler metric i.e  a smooth K\"ahler metric on $\Omega\backslash D$ and satisfies  $C^{-1}\omega_{\beta}\leq \omega\leq C\omega_{\beta}$.

\begin{thm}\label{thm weak maximal principle} (Weak Maximum Principle)
 
 Let $u$ be a $D$-subharmonic  function on $\Omega$ (in the sense of Definition \ref{def D-subharmonic function}), then 
 \[\sup_\Omega u=\sup_{\partial \Omega} u.\]
  \end{thm}
  \begin{rmk}The difference of our weak maximal principle from Jeffres' trick in \cite{Jeffres} is that our weak maximal principle applies to $L^{\infty}$-functions, while Jeffres' trick requires the function to have some H\"older continuity property near $D$.
  \end{rmk}
  \begin{proof}{of Theorem \ref{thm weak maximal principle}:}
  
Notice that the auxiliary function $\text{log }|z|$ is pluri-harmonic in $\mathbb{C}^n\backslash D$ under any K\"ahler metric, therefore $u_\epsilon=u+\epsilon \text{ log }|z|$ is also weak-subharmonic away from $D$ (smaller than harmonic lift on any ball with no intersection with $D$). However since $u$ is bounded, $u_\epsilon(p)$ goes to $-\infty$ as $p$ approaches $D\cap \Omega$. We show here that $u+\epsilon \text{ log }|z|$ can't attain interior maximum. 

If not, there exists $q\notin D$ such that 
\[u_\epsilon(q)=\sup_{\Omega}u_\epsilon . \]
Choose a ball $B_{q}$ with no intersection with $D$ and there exists some point $b\in \partial B_{q}$ such that 
\begin{equation}\label{equ u has attain small value on boundary points}
u_\epsilon(b)<\sup_{\Omega}u_\epsilon.
\end{equation}

 Then we consider the harmonic lifting of $u_\epsilon$ over $B_{q}$ as $\bar{u}_{\epsilon}$. By definition, we have 
$\bar{u}_{\epsilon}\geq u_\epsilon$. By maximal principle on $B_{q}$, we deduce $\sup\bar{u}_{\epsilon}\leq \sup u_{\epsilon}|_{\partial B_{q}}$.
Then we see that $\bar{u}_{\epsilon}$ attains interior maximum in $B_{q}$ at $q$. This means the harmonic funtion $\bar{u}_{\epsilon}$ is a constant over the whole $B_{q}$, which contradicts (\ref{equ u has attain small value on boundary points}).

Thus $u_{\epsilon}$ attain maximum on $\partial \Omega\setminus D$. 
We compute for any $p\notin D$ that 
\begin{eqnarray*}
& &u(p)
\\&=& u_{\epsilon}(p)-\epsilon \log|z|(p)
\\&\leq & \sup_{\partial \Omega\setminus D}u_{\epsilon}-\epsilon \log|z|(p)
\\&\leq & \sup_{\partial \Omega\setminus D}u+(\sup_{\partial \Omega\setminus D}\epsilon \log|z|)-\epsilon \log|z|(p)
\\&\leq & \sup_{\partial \Omega\setminus D}\varphi +\epsilon C-\epsilon \log|z|(p).
\end{eqnarray*}
Let $\epsilon \rightarrow 0$ we obtain 
\[u(p)\leq  \sup_{\partial \Omega\setminus D}\varphi.\]
Since $u\in C^{0}(\bar{\Omega}\setminus D)$, the proof  is completed. 
  \end{proof}

\section{Dirichlet problem of conical elliptic equations.}
In this section we work in the polar coordinates. 
\begin{Def}\label{def D-subharmonic function} We say $v\in C^{0}(\bar{\Omega}\setminus D)\cap L^{\infty}(\bar{\Omega})$ to be a $D$-subharmonic function if for any ball $B\in \Omega$ and $B\cap D=\emptyset$, the harmonic lifting $\bar{v}$ satisfies $\bar{v}\geq v$ in $B$. 
\end{Def}
\begin{rmk}In this section we don't require the target function to be in $C^2(\bar{\Omega}\setminus D)$, this is because we want the set of all 
$D$-subharmonic functions under the boundary condition to be closed under harmonic lifting away from $D$. This shows upper envelope  is harmonic away from $D$ and is in $L^{\infty}(\bar{B})$, then 
Lemma \ref{lem Linfty solution is actually a weak solution} can be applied. These are crucial in the proof of Theorem \ref{thm Liouville}
in section \ref{section Proof of Theorem liouville}.
\end{rmk}
\begin{thm}\label{thm Dirichlet problem for continuous boundary value}
Suppose $B$ is a ball. Let $\varphi\in C^{0}(\partial B\backslash D)\cap L^{\infty}(\partial \bar B)$ be a function defined on $\partial  B\setminus D$. Then, there exists a  $\omega$-harmonic function $u$ defined on $  B$ such that $u$ attain  the boundary value $\varphi$  continuously away from $D$. i.e 
\[\Delta_{\omega}u=0\ \textrm{over}\  B\setminus D,\]
and    for any $\xi\in \partial B\setminus D$, we have 
\[ \lim_{x\rightarrow \xi}|u(x)-\varphi(\xi)|= 0.\]

\end{thm}
\begin{thm}\label{thm Dirichlet problem for C2alpha boundary value}Suppose $B$ is a ball. Let $\varphi\in C^{2,\alpha}(\partial  B\backslash D)\cap L^{\infty}(\partial \bar B)$ be a function defined on $\partial  B\setminus D$. Then, there exists a  $\omega-$harmonic function $u$ defined on $  B$ such that $u$ attain  the boundary value $\varphi$ in Lipshitz sense away from $D$. i.e 
\[\Delta_{\omega}u=0\ \textrm{over}\  B\setminus D,\]
and    for any $\xi\in \partial\Omega\setminus D$, there exists a 
postive constant $r_{\xi}$ and $K(\xi)$ such that 
\[ |u(x)-\varphi(\xi)|\leq K(\xi)|x-\xi| \]
for all $x\in B_{\xi}(r_{\xi})\cap B$. Moreover, $u\in C^{2,\alpha}[(\bar{B}\setminus D)]$. 
\end{thm}
\begin{proof}{of Theorem \ref{thm Dirichlet problem for continuous boundary value} and \ref{thm Dirichlet problem for C2alpha boundary value}:}   Given $\varphi\in L^{\infty}( B)\cap C^{0}( B\setminus D)$, we define the value of $\varphi$ at $p\in D\cap \partial  B$ as \[\varphi(p)=\lim_{x\rightarrow p,\ p\notin D}\inf\varphi(x).\]

Define \[S_{\varphi}=\{u|u\ \textrm{is D-subharmonic and}\ u\leq \varphi\ \textrm{over}\ \partial B\setminus D\},\] 
and the upper-envelope as \[u(x)=(\sup_{u\in S_{\varphi}} u)(x),\ x\in B\setminus\ D.\] We now prove the claim
\begin{clm}\label{clm upper envelope is harmonic away from D} $\Delta_{\omega} u=0$ over $ B\setminus D$.
\end{clm}
This goes exactly as in \cite{GT}, except the harmonicity holds only over $ B\setminus D$ and the harmonic-lifting are performed away from $B\setminus D$. For the reader's convenience we include the crucial detail here. 
    Suppose $p\notin D$ and $\lim_{k\rightarrow \infty} v_{k}(p)\rightarrow u(p)$, we  choose $B_{p}(R)$ with no intersection with $D$. We consider the harmonic lifting of $v_{k}$ in $B_{p}(R)$  as $\bar{v}_{k}$. Then 
    \begin{equation}\label{equ relation of upper envelope and bar v}
    \lim_{k\rightarrow \infty}\bar{v}_{k}=\bar{v}\ \textrm{over}\ B_{p}(\frac{R}{2}),\ u(p)=\bar{v}(p).
    \end{equation}
    It suffices to show $\bar{v}\equiv u$ over $B_{p}(\frac{R}{4})$. If not, there exists a $q\in B_{p}(\frac{R}{4})$ such that $\bar{v}(q)\neq u(q)$. Then there exists a $\widehat{u}\in S_{\varphi}$ such that 
    \begin{equation}
    v(q)<\widehat{u}(q)\leq u(q).
    \end{equation}
    
    Now we refine the sequence $v_{k}$ by considering $max(v_{k},\widehat{u})$ and denote their harmonic lifting over $B(R)$ as $w_{k}$. Then we have 
    \[v_{k}\leq w_{k}\leq u\ \textrm{in}\ B_{p}(\frac{3R}{5}).\] 
    Then let $k\rightarrow \infty$, $w_{k}\rightarrow w_{\infty}$ over 
    $B_{p}(\frac{R}{2})$. Then by maximal principle we have:
    \begin{equation}
    \bar{v}_{\infty}\leq w_{\infty}\leq u\ \textrm{in}\ B_{p}(\frac{R}{2}).
    \end{equation}
    
    Both $\bar{v}_{\infty}$ and $w_{\infty}$ are harmonic. We have 
    \[\bar{v}_{\infty}(q)<\widehat{u}(q),\ \textrm{but}\ \bar{v}_{\infty}(p)=\widehat{u}(p),\]
    This is a contradiction since by strong maximal principle over $B_{p}(R)$ which does not intersect $D$, we have $\bar{v}_{\infty}\equiv w_{\infty}$ over $B_{R}$. The proof of Claim \ref{clm upper envelope is harmonic away from D} is complete.
    
   On the attainability of the boundary value, since the domain we consider is a ball, which is convex,   we choose the barriers at those $p\in \partial B\setminus D$  exactly as  in 
  formula (6.45) of \cite{GT}, with $\tau=1$ and $R$ small enough such that $B_{p}(10R)\cap D=\emptyset$. Then boundary value $\varphi 
  $ is then attained continuously away from $D$. Thus the proof of 
  Theorem \ref{thm Dirichlet problem for continuous boundary value} is
  complete.
  
  Suppose $\varphi\in C^{2,\alpha}(\partial  B\backslash D)\cap L^{\infty}(\partial \bar B)$.   The boundary value $\varphi$ is attainable in Lipshitz-sense  together with the fact that $u\in C^{2,\alpha}(\bar{B}\setminus D)$ are
   trivially implied by the proof of Theorem 6.14 in \cite{GT}. The proof Theorem \ref{thm Dirichlet problem for C2alpha boundary value} is complete.
    \end{proof}
   
   \section{Strong Maximum Principles and Trudinger's estimate. \label{section Strong Maximum Principles and De Giorgi estimate}}
 In this section we prove the strong maximum principle, which is crucial in the proof of Theorem \ref{thm Liouville} in section \ref{section Proof of Theorem liouville}.
 
 Let us  first recall the classical strong maximum principle in Euclidean space, which states the subharmonic function which takes interior supremum would be a constant function. The following main theorem of this section is a generalization of the classical strong maximum principle.
    \begin{thm}\label{thm strong maximal principle}(Strong Maximum Principle)
Let $u\in C^2(\Omega\backslash D)\cap L^{\infty}(\bar\Omega)$ be a bounded real value function on $\Omega$ which satisfies $\Delta_\omega u\geq 0$ on $\Omega\backslash D$, suppose there exists a smaller subdomain $\Omega'\subset\subset \Omega$ such that $\sup_\Omega u=\sup_{\Omega'}u$, then $u$ must be a constant.
\end{thm}
 
 \begin{proof}{of Theorem \ref{thm strong maximal principle}:}
 
 We have two proofs for this fact. One is a barrier construction and the other one is also straight forward by Trudinger's estimate. Since both proofs have their own interest, we include both of them here.
 
 Proof $1$: Barrier construction. 
 
 By localizing the supremum of $u$ in the interior $\Omega\cap D$, there is no loss of generality in assuming that we are in a situation where $\Omega=B_2=\{(z,z_2,\cdots,z_n)||z|^2+|z_2|^2+\cdots+|z_n|^2\leq 4\}$, and $\sup_\Omega u=1=\lim_{i\to \infty}u(p_i)$ for a sequence of points $p_i\to o$. 
 
 We prove  by contradiction. Suppose $u$ is not constant function 1,  by  the classical strong maximum principle in the smooth case, we know $u<1$ in  $\{ {1\over 2}\leq |z_1|\leq 1\}$. Then we suppose $u\leq \tau_{0}<1$ in  the ring-shaped piece $R_\delta=\{|z|^2+|z_2|^2+\cdots+|z_n|^2=1, |z_2|^2+\cdots+|z_n|^2\leq \delta^2\}\subset\partial B_1$ for some definite number $\delta>0$ (this could be done by continuity of $u$), and $\sup_{\partial B_1}u=1$.
 
 Let $\psi=\frac{\delta^2}{2}(1-\delta^2)^{-\frac{\beta}{2}}|z|^\beta-(|z_2|^2+\cdots+|z_n|^2)$, then:
  
 \begin{itemize}
 \item On $R_\delta$, $\psi\leq \frac{\delta^2}{2}(1-\delta^2)^{-\frac{\beta}{2}};$

 \item  On $(\partial B_1)\backslash R_\delta$, $\psi\leq-\frac{\delta^2}{2}<0;$
 \item On $B_1\backslash D$, $\Delta_\omega \psi\geq -(n-1)C. $ Since $\psi$ is bounded function, so it is also a global weak subharmonic function by virtue of  Lemma 1.1;
 \item On $B_1\backslash D$, $|\nabla \psi|_\omega^2\geq C^{-1}|\nabla \psi|_{\omega_{(\beta)}}^2\geq C^{-1}\beta^{-2}|z|^{2-2\beta}|\frac{\partial \psi}{\partial z}|^2\frac{1}{4}\delta^4(1-\delta^2)^{-\beta}=\frac{C^{-1}}{16}\delta^4(1-\delta^2)^{-\beta}$.
 \end{itemize}
  
 It follows that $$\Delta_\omega e^{a\psi}=(a^2|\nabla \psi|_\omega^2+a\Delta_\omega\psi)e^{a\psi}\geq (a^2\frac{C^{-1}}{16}\delta^4(1-\delta^2)^{-\beta}-a(n-1)C)e^{a\psi}\geq 0,$$ for $a\geq 16 \delta^{-4}(1-\delta^2)^\beta(n-1)$.
  
  Therefore, for the ``bumped function'' $u_\epsilon=u+\epsilon(e^{a\psi}-1)$ is a bounded subharmonic function on $B_1\backslash D$ since $\Delta_{\omega}u_\epsilon=\Delta_\omega u+\epsilon \Delta_\omega e^{a\psi}\geq 0$. 
  
  On the other hand, $\lim_{i\to \infty}u_\epsilon(p_i)=\lim_{i\to \infty} \{u(p_i)+\epsilon (e^{a\psi(p_i)}-1)\}=1$, 
  while on the boundary:
  
  \begin{itemize}
  \item $\sup_{(\partial B_1)\backslash R_\delta} u_\epsilon \leq \sup_{\partial B_1}\{u+\epsilon (e^{-\frac{\delta^2}{2}}-1)\}\leq 1-\epsilon (1-e^{-\delta^2/2});$ 
  \item $\sup_{R_\delta} u_\epsilon \leq \sup_{R_\delta}u+\epsilon (e^{a/2 \delta^2(1-\delta^2)^{-\frac{\beta}{2}}}-1)\leq \tau_{0}
  +\epsilon (e^{a/2 \delta^2(1-\delta^2)^{-\frac{\beta}{2}}}-1).$
  \end{itemize}
  
 Thus, since $\tau_{0}<0$,   by taking $\epsilon>0$ small enough, we can make $\sup_{\partial B_1}u_\epsilon <1\leq \sup_{B_1}u_\epsilon$, which contradicts the weak maximal principle in Theorem \ref{thm weak maximal principle}.\\

Proof 2: Trudinger's Harnack inequality.   Without loss of generality, we can still
assume $u$ attains interior maximum at $0$. Suppose $u$ is not a constant, then there exists a ball $B_{0}(r_0)$ such that $B_{0}(2r_0)\in \Omega$, and $u\neq u(0)$ at some point in $\partial B_{0}(r_0)$.   Using Theorem \ref{thm Dirichlet problem for continuous boundary value}, we can find a solution $v\in C^{0}(\bar{\Omega}\setminus D)\cap L^{\infty}(\bar{\Omega})$ to the following equation 
\begin{equation}
\Delta_{\omega}v=0\ \textrm{in}\ B_{0}(r_{0}),\ v|_{\partial \Omega}=u| _{\partial \Omega}.
\end{equation} 

By the weak maximal principle in Theorem \ref{thm weak maximal principle}, we have 
\begin{equation}
v(p)\leq  \sup_{\partial B_{0}(r_0)}v=\sup_{B_{0}(r_0)}v\leq u(0)\ \textrm{and}\  v(p)\geq u(p),\ \textrm{for all}\
p \in   B_{0}(r_0).
\end{equation}  
This  means $v$ also attains interior maximum at $0$. 
Using the Trudinger's maximal principle in Proposition \ref{prop DeGiorgi strong maximal princle} (actually we only need the Harnack inequality to be true for some $p_0>0$  for the proof the strong maximal principle), $v\equiv v(0)=u(0)$ is a constant.
This constradicts the hypothesis that   $v(q)=u(q) \neq u(0)$ at some point $q\in \partial B_{0}(r_0)$.

 \end{proof}

We work in the polar coordinates to reformulate the De-Giorge estimate in Theorem 8.18 of \cite{GT} in the following proposition.
\begin{prop}\label{prop DeGiorgi strong maximal princle}(Trudinger's stong maximal principle) Suppose $\omega$ is a weak-conical metric over $B(1)$. Suppose $u\in L^{\infty}(B(1))\cap C^{2}(B(1)\setminus D)$ is a $\omega$-harmonic function i.e
\begin{equation}\label{equ omega harmonic equation}\Delta_{\omega}u=0\ \textrm{in}\ B(1).
\end{equation}
Suppose there exists a ball $B_{p}(r_0)\in {B}(1)$, such that $
u(p)=\sup_{B_{p}(r_0)}u$ or $
u(p)=\inf_{B_{p}(r_0)}u$. Then $u$ is a constant over $B(1)$.
\end{prop}
\begin{proof}{of Proposition \ref{prop DeGiorgi strong maximal princle}:}
    This is a directly corollary of Lemma \ref{lem Trudinger's Harnack inequality}. We just prove the case when  $
u(p)=\sup_{B_{p}(r_0)}u$. Since 
$u\in L^{\infty}(B(1) \setminus D)$, then by Lemma \ref{lem Linfty solution is actually a weak solution}, $u$ is a weak solution to 
(\ref{equ omega harmonic equation}).  Let $v=u(p)-u=(\sup_{B_{p}(r_0)}u)-u$, then $v\geq 0$ in $B_{p}(r_0)$. Then using 
Lemma \ref{lem Trudinger's Harnack inequality}, for some $q>0$ (this is all we need, though Lemma \ref{lem Trudinger's Harnack inequality} says more than this), we have
\begin{equation}\label{equ applying Degiorg estimate to v}|v|_{L^{q}, B_{p}(\frac{r_0}{10})}\leq Cr_{0}^{2n}\inf_{B_{p}(\frac{r_0}{20})} v.
\end{equation}
Using $v(p)=0$ and the $C^{\alpha}$-continuity of $v$ from Lemma \ref{lem Linfty solution is actually a weak solution}, we get $\inf_{B_{p}(\frac{r_0}{10})} v=0$. Hence  (\ref{equ applying Degiorg estimate to v}) implies 
\[|v|_{L^{q}, B_{p}(\frac{r_0}{10})}=0,\]
which means $v=0$ in $B_{p}(\frac{r_0}{10})$. This implies $v\equiv 0$ over $B(1)$, which means $u$ is a constant. 
\end{proof}

\section{Proof of Theorem \ref{thm Liouville}. \label{section Proof of Theorem liouville}}
 Consider $tr_{\omega_{\beta}}\omega$.  
Given the $\omega$ and $\phi$ as in Theorem \ref{thm Liouville}, we define the 3rd derivative as  
\begin{equation}\label{equ Def of 3rd derivative quantity S}S=\omega^{i\bar{j}}\omega^{s\bar{t}}\omega^{p\bar{q}}\phi_{z_{i}, \bar{z}_{t},z_{p}}\phi_{\bar{z}_{j}, z_{s},\bar{z}_{q}},
\end{equation}
as in \cite{Yau} and \cite{Brendle}.
The derivatives concerned are all covariant derivatives with respect to $\omega_{\beta}$. Nevertheless, since the connection of $\omega_{\beta}$ is holomorphic, we have \[\phi_{z_{i}, \bar{z}_{t}}=\frac{\partial^2 \phi}{\partial z_{i} \bar{\partial} \bar{z}_{t}}=\omega_{ z_{i}, \bar{z}_{t}}.\]
Thus $S$ is actually defined over the whole $C^{n}$, without assuming the existence of a global potential $\phi$. By equation (2.7) in \cite{Yau}, we have 
\begin{equation}\Delta_{\omega}tr_{\omega_{\beta}}\omega\geq \frac{S}{K}\geq 0,\ S\ \textrm{as in}\ (\ref{equ Def of 3rd derivative quantity S}).
\end{equation}
 Without loss of generality, we may assume that
\[
{1\over {C_0}} \leq tr_{\omega_{\beta}}\omega \leq  C_0
\]and
\[
\sup_{\C\times \C^{n-1}}\; tr_{\omega_{\beta}}\omega =  C_0.
\]

Since $tr_{\omega_{\beta}}\omega$ is subharmoic, if this sup is achieved in some finite ball, then the strong maximal principle (Theorem  \ref{thm strong maximal principle}) implies that
\[
tr_{\omega_{\beta}}\omega =  constant.
\]
Going back to (36), we see that $\omega$ is covariant constant with respect to $\omega_{\beta}$. This easily implies that $\omega$ is isometric to $\omega_{\beta}$ by a complex linear transformation.

  Unfortunately,  a bounded function will usually not achieve maximum at an interior point.  Suppose
 \[
\sup_{\C\times \C^{n-1}} \;  tr_{\omega_{\beta}}\omega = C_0.
 \]     
Suppose there exists a sequence of points $p_i$ such that 
\[
  tr_{\omega_{\beta}}\omega(p_{i}) \rightarrow C_0,\  dist(p_{i},0)\rightarrow \infty,\qquad {\rm as}  \;i \rightarrow \infty.
\]
Consider the rescaled sequence $(\C\times \C^{n-1}, 0, \omega_i = R_i^{-2} \omega).\;$ It converges locally smoothly to 
 $(\C\times \C^{n-1}, o, \omega_\infty).\;$   Denote
 \[
   v_i = tr_{\omega_{\beta}}\omega_{i}  \in [{1\over C_0},C_0]\qquad {\rm in }\; \C\times \C^{n-1}.
 \]
 Then 
 \[
 v_i(p_{i}) \rightarrow C_{0},\qquad {\rm as} \; \; i \rightarrow \infty.
 \]
 It is easy to see that $(p_{i}, v_i)$  converges to $(p_{\infty}, v_\infty)$
locally smooth away from divisor such that
\[
v_\infty  = tr_{\omega_{\beta}} \;\omega_\infty \in [{1\over C_0},C_0] \qquad {\rm in }\; \C\times \C^{n-1},\]
and
\begin{equation}\label{equ Yau's bochner technique after taking limit}
\triangle_{\omega_\infty} v_\infty =  S_{\infty} \geq 0,\ dist_{\beta}(p_{\infty},0)=1.
\end{equation}
If $|v|_{\alpha,\beta}\leq C$ before taking limit, then $v_{\infty}$ achieves interior maximum at $p_{\infty}$,  from Theorem \ref{thm strong maximal principle} we obtain $v_\infty$ is a constant. By (\ref{equ Yau's bochner technique after taking limit}) we deduce $S_{\infty}=0$. Therefore,  
\[
\omega_\infty = L^{\star}\omega_{\beta},
\]
for some linear transformation $L$.
Then, by Theorem \ref{thm Tangent cone trivial implies omega trivial}, we know that  $\omega=L^{\star}\omega_{\beta}$.

So the difficulty is to show $v_{\infty}$ is a constant even it might not be continous apriorily.   Fortunately, $v_{\infty}$ is apprximated by the 
sequence $v_{i}$. It is here
we apply  harmonic lifting  before letting $i\rightarrow \infty$.

In the singular polar coordinates, we consider the ball centered at $0$ and with radius $2$.  By Theorem \ref{thm Dirichlet problem for C2alpha boundary value}, we can find a $\omega_{i}$-harmonic function $h_{i}$ such that 
\[
\triangle_{\omega_i} h_{i} = 0, \qquad {\rm in}\; B_2(o)
\]
and 
\[
h_{i} = v_i \qquad {\rm at }\; \partial B_2(o).
\]
Using weak maximal principle,  we have  $h_{i} \geq v_i\geq 0$. 
Moreover, since $v_i$ is bounded above by $C_{0}$ in the boundary, It follows by maximum principle again that $h_i \leq C_0$ in $B_2.\;$
Thus $0\leq h_{i}\leq C_{0}$. 
It follows that
\[
 h_{i}(p_{i}) \rightarrow C_{0},\qquad {\rm as} \; \; i \rightarrow \infty.
\]
 By Lemma \ref{lem Linfty solution is actually a weak solution}, we know that $h_{i}$ is uniformly $C^{\alpha, \beta}$ in the interior
and continuous up to all smooth points on $\partial B_2(o) \setminus D.\;$ 

 Now we take limit as $i\rightarrow \infty$, and denote the limit of $h_{i}$ as
$h_\infty.\;$   The convergence is locally smooth away from divisor, uniformly $C^{\alpha,\beta}$ across the divisor. Thus, we have
\[
\triangle_{\omega_\infty} h_\infty = 0,  \qquad {\rm in}\; B_2(o, \omega_{\beta})
\]
and 
\[
    \; h_\infty(p_{\infty}) = C_0.
\]
Applying strong maximal principle theorem (Theorem 3.3), we have
\[
h_\infty \equiv C_{0} \qquad {\rm  in}\;\; B_2(o).
\]
Moreover, $v_{i}\rightarrow v_{\infty}$ smoothly away from $D$ and consequently $h_{\infty}|_{\partial B_{2}}=v_{\infty}|_{\partial B_{2}}$.

It follows that, on  $\partial B_2(o)\setminus D$, we have 
\[
v_\infty  = h_\infty \equiv C_0.
\]
Then $v_{\infty}$ attains maximum over $\partial B_2(o)\setminus D$!
Using the subharmoncity in (\ref{equ Yau's bochner technique after taking limit}) and  strong maximal principle again, we deduce $v_{\infty}$ is a constant and consequently
\[
 S_{\infty} \equiv 0.
\]
Hence $ \omega_\infty=L^{\star}\omega_{\beta}$. Since $\omega_{\infty}$ is a tangent cone of $\omega$, using  Theorem \ref{thm Tangent cone trivial implies omega trivial}, we conclude   $\omega=L^{\star}\omega_{\beta}$. 

The proof of Theorem \ref{thm Liouville} is complete.

\section{Bootstrapping of the conical K\"ahler-Ricci flow.\label{section Bootstrapping of  CKRF}}
In this section we show the bootstrapping of conical K\"ahler-Ricci flow
is true. This is important when we show the convergence of the rescaled
sequence in the proof of Theorem \ref{long time existence of CKRF over Riemann surface}.  
\begin{thm}\label{Bootstrapping of Holder metrics}Suppose $\alpha>0$  and $\phi$ is a $C^{2+\alpha,1+\frac{\alpha}{2},\beta}$ solution to the conical K\"ahler-Ricci flow over $[0,t_0]$, then $\phi\in C^{2+\acute{\alpha},1+\frac{\acute{\alpha}}{2},\beta} $ for all $\acute{\alpha}<\min\{\frac{1}{\beta}-1,1\}$ when $t>0$. Moreover there exists a constant $C(|\phi|_{2+\alpha,1+\frac{\alpha}{2},\beta,M\times [0,t_0]})$ (depending on $|\phi|_{2+\alpha,1+\frac{\alpha}{2},\beta}$, $\acute{\alpha}$,  $g_0$, and the data in Definition \ref{Convention on the constant} ) such that \[|\phi|^{(\star)}_{2+\acute{\alpha},1+\frac{\acute{\alpha}}{2},\beta,M\times [0,t_0]}\leq C(|\phi|_{2+\alpha,1+\frac{\alpha}{2},\beta,M\times [0,t_0]}).\]
\end{thm}
\begin{proof}{of Theorem \ref{Bootstrapping of Holder metrics}:}
Temporarily we denote $|\phi|_{2+\alpha,1+\frac{\alpha}{2},\beta,M\times [0,t_0]}=\underline{k}$. Let  $u_i$ be a tangential variable near $D$. Differentiating
the CKRF (\ref{ckrf potential equation}) with respect to $u_i$ we get  
\begin{equation}\frac{\partial \phi_{u_i}}{\partial t}=\Delta_{\phi}\phi_{u_i}+\beta \phi_{u_i}+\widehat{h}\ \textrm{over}\ B_{0}(r_0),
\end{equation}
where $\widehat{h}$ is a $C^{\alpha,\frac{\alpha}{2},\beta}$ function and $r_0$ is sufficient small such that a coordinate exists in $B_{0}(r_0)$ .
Then exactly as in the proof of Theorem 1.13 in \cite{WYQWF}, by applying  the interior parabolic Schauder estimate in the equation (21) in \cite{CYW}, we obtain 
\begin{equation}\label{Bounding C 2 alpha norm of the tangential derivatives}
|\phi_{u_i}|^{(\star)}_{2+\alpha,1+\frac{\alpha}{2},\beta,B_{0}(r_0)\times [0,t_0]}\leq C\times (1+|\phi_{u_i}|_{0,B_{0}(r_0)\times [0,t_0]})= C(\underline{k}).
\end{equation}
 First we bound the spatial $C^{2,\alpha,\beta}$ norm when $t>0$. Using the intepolation inequalities in Lemma 11.3 in \cite{CYW} for $\phi_{u_i}$, we end up with 
 \begin{equation}
 |\phi_{u_i}|^{(\star)}_{1,\alpha^{\prime},\beta,B_{0}(r_0)\times [0,t_0]}\leq C\ \textrm{for any}\ \acute{\alpha}<\min\{\frac{1}{\beta}-1,1\}.
 \end{equation}
Hence the mixed derivatives and tangential second order derivatives  satisfy
   \begin{equation}\label{Bounding mixed and tangential derivatives I}
  |\phi_{\bar{\mathfrak{a}}u_i}|_{\alpha^{\prime},\beta,B_{0}(\frac{r_0}{2})\times \{t\}}\leq \frac{C}{t^{\frac{1+\alpha^{\prime}}{2}}},
   \end{equation}
   \begin{equation}\label{Bounding mixed and tangential derivatives II}
  |\phi_{u_i\bar{u_j}}|_{\alpha^{\prime},\beta,B_{0}(\frac{r_0}{2})\times \{t\}}\leq \frac{C}{t^{\frac{1+\alpha^{\prime}}{2}}}.
   \end{equation}
   Similarly we have 
   \begin{equation}\label{Bounding mixed and tangential derivatives III}
  |\phi_{\mathfrak{a}\bar{u_i}}|_{\alpha^{\prime},\beta,B_{0}(\frac{r_0}{2})\times \{t\}}\leq \frac{C}{t^{\frac{1+\alpha^{\prime}}{2}}}.
   \end{equation}
   Thus to prove the bootstrapping estimate for $i\partial \bar{\partial} \phi$, it suffices to prove it for $\phi_{\mathfrak{a}\bar{\mathfrak{a}}}$. The key thing is that the CKRF equation (\ref{ckrf potential equation})  directly implies the bound for $\phi_{\mathfrak{a}\bar{\mathfrak{a}}}$.  Without loss of generality we assume $n=2$. Then the CKRF equation reads as 
   \begin{eqnarray*}
   & &(\omega_{D,\mathfrak{a}\bar{\mathfrak{a}}}+\phi_{\mathfrak{a}\bar{\mathfrak{a}}})(\omega_{D,u\bar{u}}+\phi_{u\bar{u}})
   -(\omega_{D,\mathfrak{a}\bar{u}}+\phi_{\mathfrak{a}\bar{u}})(\omega_{D,u\bar{\mathfrak{a}}}+\phi_{u\bar{\mathfrak{a}}})
   \\&=&e^{h-\beta\phi+\frac{\partial \phi}{\partial t}}\omega^2_D,
   \end{eqnarray*}
   where $\phi_{\mathfrak{a}\bar{\mathfrak{a}}}=(\frac{\partial^2}{\partial r^2}+\frac{1}{r}\frac{\partial}{\partial r}+\frac{1}{\beta^2r^2}\frac{\partial^2}{\partial \theta^2})\phi$. 
   Then we obtain  \begin{eqnarray}\label{Equation for the normal 1-1 derivative of the potential}
   \phi_{\mathfrak{a}\bar{\mathfrak{a}}}=\frac{e^{h-\beta\phi+\frac{\partial \phi}{\partial t}}\omega^2_D+(\omega_{D,\mathfrak{a}\bar{u}}+\phi_{\mathfrak{a}\bar{u}})(\omega_{D,u\bar{\mathfrak{a}}}+\phi_{u\bar{\mathfrak{a}}})}{\omega_{D,u\bar{u}}+\phi_{u\bar{u}}}-\omega_{D,\mathfrak{a}\bar{\mathfrak{a}}}.
   \end{eqnarray}
   
   By Theorem 1.13 in \cite{WYQWF} and intepolation,  we deduce 
   \begin{equation}\label{Bounding Calphaprime norm of the Ricci potential}
  | \frac{\partial \phi}{\partial t}|^{(\star)}_{\alpha^{\prime},\beta,M\times [0,t_0]}\leq C.
   \end{equation}
   Then by (\ref{Bounding mixed and tangential derivatives I}), (\ref{Bounding mixed and tangential derivatives II}),(\ref{Bounding mixed and tangential derivatives III}), (\ref{Equation for the normal 1-1 derivative of the potential}), and (\ref{Bounding Calphaprime norm of the Ricci potential}), we conclude 
     \begin{equation}\label{Bounding the C alpha norm of normal 1,1 derivatives}
  |\phi_{\mathfrak{a}\bar{\mathfrak{a}}}|_{\alpha^{\prime},\beta,B_{0}(\frac{r_0}{2})\times \{t\}}\leq \frac{C}{t^{\frac{1+\alpha^{\prime}}{2}}}.
   \end{equation}
   The estimates for the time derivatives and timewise H\"older norms are similar. To be simple, using (\ref{Bounding C 2 alpha norm of the tangential derivatives}) and timewise intepolation, we can get similar estimate as  follows 
    \begin{eqnarray}\label{Bounding time holder norm of mixed and tangential derivatives I}
 \nonumber & &|\phi_{\bar{\mathfrak{a}}u_i}|_{0,\frac{\alpha^{\prime}}{2},\beta,B_{0}(\frac{r_0}{2})\times [t,t_0]}+|\phi_{\bar{\mathfrak{u}_i}\mathfrak{a}}|_{0,\frac{\alpha^{\prime}}{2},\beta,B_{0}(\frac{r_0}{2})\times [t,t_0]}+|\phi_{u_i\bar{u_j}}|_{0,\frac{\alpha^{\prime}}{2},\beta,B_{0}(\frac{r_0}{2})\times [t,t_0]}
  \\&\leq& \frac{C}{t^{\frac{1+\alpha^{\prime}}{2}}}.
   \end{eqnarray}
   Thus, using (\ref{Equation for the normal 1-1 derivative of the potential}) we can bound $|\phi_{\mathfrak{a}\bar{\mathfrak{a}}}|_{0,\frac{\alpha^{\prime}}{2},\beta,B_{0}(r_0)\times [t,t_0]}$ exactly as how we get (\ref{Bounding the C alpha norm of normal 1,1 derivatives}). 
   \\
   
   The proof is complete. Actually what we proved is with better weight than what's stated in Theorem \ref{Bootstrapping of Holder metrics}. 
\end{proof}

In particular, with respect to the bootstrapping of conical K\"ahler-Einstein metrics, we've recovered a result of Chen-Donaldson-Sun  in \cite{CDS2}.
\begin{thm}( Chen-Donaldson-Sun): Suppose $\phi$ is a conical K\"ahler -Einstein metric and $\phi\in C^{2,\alpha,\beta}$ for some $\alpha>0$. Then $\phi\in C^{2+\acute{\alpha},1+\frac{\acute{\alpha}}{2},\beta} $ for all $\acute{\alpha}<\min\{\frac{1}{\beta}-1,1\}$ and 
\[|\phi|_{2,\acute{\alpha},\beta,M}\leq C(|\phi|_{2,\alpha,\beta,M}).\]
\end{thm}
\section{Exponential convergence when $C_{1,\beta}<0$ or $=0$.\label{section Exponential Convergence when C_1 nonpositive} }
In this section, we prove Theorem \ref{Convergence to KE metric when C1 nonpositive} on the convergence of CKRF. We follow the proof  of Cao \cite{Cao} and employ some modifications which are necessary in the conical case at this point. \\

We  point out a convention of notations in this section: The $C$'s in this section are all time independent constants, the other dependence of the $C's$ in this section is as Definition \ref{Convention on the constant}. 
\begin{proof}{of Theorem \ref{Convergence to KE metric when C1 nonpositive}:} We only prove the case when   $C_{1,\beta}=0$, since the case when $C_{1,\beta}<0$ is much much easier and doesn't require any other machinery except maximal principle of the heat equation and Theorem 1.8 in \cite{CYW}. \\

By Theorem Theorem 1.13 in \cite{WYQWF}, we know $Ric$ and $\sqrt{-1}\partial \bar{\partial} \frac{\partial \phi}{\partial t}$ are   $C^{\alpha,\beta}$ (1,1)-forms. Moreover, the scalar curvature $s_{\phi}$ and $\nabla \frac{\partial \phi}{\partial t}$ are all in $C^{\alpha,\beta}$. Then, using  regularity of lower order items establised in \cite{CYW},   the identities in the following proof  are all well defined. 

 In the Calabi-Yau case,  there is a smooth function $h_{\omega_0}$ such that 
\begin{equation}\label{Background equation when C1beta = 0}
Ric_{\omega_0}=i\partial \bar{\partial}\{h_{\omega_0}-(1-\beta)\log h\},\  
Ric_{\omega_0}-2\pi(1-\beta)[D]=i\partial \bar{\partial}H_{\beta},
\end{equation}
where 
\begin{equation*}
H_{\beta}=h_{\omega_0}
-(1-\beta)\log|S|^2
\end{equation*}
and $h$ is the metric of the line bundle $L_{D}$. The potential equation of the Calabi-Yau CKRF reads as 
\begin{equation}\label{potential equation of the Calabi-Yau CKRF}
(\omega_{D}+i\partial \bar{\partial}\phi)^n=e^{-h_{\omega_D}+\frac{\partial \phi}{\partial t}}\omega^n_{D}.
\end{equation}

Step 1. The most important thing is to obtain a time-independent bound for $osc\phi$. This is achieved similarly as in  \cite{Cao}, the difference is that we apply the Poincare inequality here, while in \cite{Cao} the lower bound on the Green function is applied. Notice  $\frac{\partial \phi}{\partial t}$ satisfies
\begin{equation}\label{Calabi-Yau: Evolution of dphidt}
\frac{\partial}{\partial t}\frac{\partial \phi}{\partial t}=\Delta_{\phi}\frac{\partial \phi}{\partial t}.
\end{equation}
By maximal principle we obtain 
\begin{equation}\label{Bound on the dphidt in the Calabi-Yau case}
|\frac{\partial \phi}{\partial t}|_{0,[0,\infty)}\leq C.
\end{equation}
The from the Calabi-Yau CKRF equation we get 
\begin{equation}
(\omega_{D}+i\partial \bar{\partial}\phi)^n=e^{F(t)}\omega^n_{D}\ \textrm{for}\ |F(t)|_{0,[0,\infty)}\leq C.
\end{equation}
Hence, by considering $(\omega_{D}+i\partial \bar{\partial}\phi)^n-\omega_{D}^n$, we compute
\begin{equation}\label{Difference of the volume forms}
i\partial \bar{\partial}\phi\wedge (\omega_{\phi}^{n-1}+.....+\omega_{D}^{n-1})=[e^{F(t)}-1]\omega^n_{D}.
\end{equation}
Now we take $\phi_0=\phi-\underline{\phi}$ so that 
the average of $\phi_0$ with respect to $\omega_{D}$ is $0$. Then we multiply (\ref{Difference of the volume forms}) by $\phi_0$ and integrate over $M$ we get
\begin{equation}
-\int_{M}\partial \phi_0\wedge \bar{\partial}\phi_0\wedge (\omega_{\phi}^{n-1}+.....+\omega_{D}^{n-1})=\int_{M}\phi_{0}[e^{F(t)}-1]\omega^n_{D}\leq C.
\end{equation}
 Notice that every form in the parenthesis on the left hand side is positive, we obtain 
\begin{equation}
\int_{M}|\nabla_{\omega_{D}}\phi_0|^2 \omega_{D}^{n}=n\int_{M}
\partial \phi_{0}\wedge \bar{\partial}\phi_{0}\wedge \omega_{D}^{n-1}\leq C\int_{M}|\phi_0|\omega^n_{D}.
\end{equation}
By the Poincare inequality for $\omega_{D}$ (stated in Remark 4.4 in \cite{WYQWF}), and the assumption  
$\frac{1}{Vol(M)}\int_{M}\phi_0\omega_{D}=0$, we obtain 
\begin{equation}
\int_{M}\phi_0^2\omega_{D}^{n}\leq C\int_{M}|\nabla_{\omega_{D}}\phi_0|^2 \omega_{D}^{n}\leq C\int_{M}|\phi_{0}|\omega^n_{D}\leq C+\frac{1}{100}\int_{M}\phi_0^2\omega_{D}^{n}.
\end{equation}
Therefore we obtain 
\begin{equation}\label{L^2 bound on the potential phi ind of time}
\int_{M}\phi_{0}^2\omega_{D}^{n}\leq C,
\end{equation}
which is the necessary $L^2$-bound in the Moser iteration scheme. \\

 Let
\[\phi_{0,+}=\max\{\phi_0,0\},\ \phi_{0,-}=-\min\{\phi_0,0\}.\]
Notice that both $\phi_{0,+}$ and $\phi_{0,-}$ are nonnegative. Lemma 7.6 of \cite{GT} and the existence of singular coordinate near $D$  immediately implies both $\phi_{0,+}$ and $\phi_{0,-}$ are Lipshitz functions with respect to $\omega_{D}$. Thus for any $p>1$, we can also multiply equation (\ref{Difference of the volume forms}) by $\phi^p_{0,+}$ ($\phi^p_{0,-}$) and apply Lemma 7.6 of \cite{GT} to get 
\begin{equation}
-p\int_{M}\phi^{p-1}_{0,+} \partial \phi_{0,+} \wedge \bar{\partial}\phi_{0,+} \wedge (\omega_{\phi}^{n-1}+.....+\omega_{D}^{n-1})=\int_{M}\phi^{p}_{0,+} [e^{F(t)}-1]\omega^n_{D}.
\end{equation}
Thus we obtain 
\begin{equation}
\int_{M} |\nabla_{\omega_{D}} \phi_{0,+}^{\frac{p+1}{2}}|^2\omega_{D}^{n}
\leq \frac{C(p+1)^2}{4p}\int_{M}\phi^{p}_{0,+}\omega^n_{D}.
\end{equation}
By the Sobolev constant bound (see Remark 4.4 in \cite{WYQWF}) and (\ref{L^2 bound on the potential phi ind of time}), the Moser's iteration as in \cite{Cao} works and we  obtain the time-independent bound on $\phi_{0,+}$:
\begin{equation}
|\phi_{0,+}|_{0,[0,\infty)}\leq C.
\end{equation}
In the same way we get $|\phi_{0,-}|_{0,[0,\infty)}\leq C$. Thus finally we completed step 1 by obtaining
\begin{equation}\label{Bound on the ocsillation of phi}
osc \phi\leq C.
\end{equation}

Step 2.  By the proof of Proposition \ref{C^2 estimate}, the equation (\ref{Bound on the dphidt in the Calabi-Yau case}), and (\ref{Bound on the ocsillation of phi}),  we obtain
\begin{equation}\label{Calabi-Yau: C11 bound}
\frac{C}{\omega_{D}}\leq \omega_{\phi}\leq C\omega_{D}.
\end{equation}
Therefore by the last part of the proof of Theorem \ref{Evans Krylov} (on the norm dependence, section \ref{section of Holder estimate for the second derivative}), and equation (\ref{potential equation of the Calabi-Yau CKRF}) (which does not concern any $0$th order term of $\phi$ on the right hand side), we obtain
 \begin{equation}\label{Calabi-Yau: bound on C2alpha beta norm of the potential}
|i\partial \bar{\partial}\phi|_{\alpha,\beta,[0,\infty)}\leq C.
\end{equation}
Thus the $C^{\alpha,\beta}$ norm of $\omega_{\phi}$ is bounded independent of time and  any sequence $\omega_{\phi_{t_k}}$ at least  subconverges to a limit $\omega_{CY,\infty}$. Furthermore, by (\ref{Calabi-Yau: Evolution of dphidt}), Theorem 1.18 in \cite{CYW}, and (\ref{Bound on the dphidt in the Calabi-Yau case}), we obtain
\begin{equation}\label{Calabi-Yau: Bound on the C2alpha norm of dphidt}|\frac{\partial \phi}{\partial t}|_{2,\alpha,\beta,[t,\infty)}\leq C(t),\ C(t)<\infty\ \textrm{when}\ t>0. \end{equation}

Step 3. In this step we prove the flow subconverges to a Ricci-Flat metric  to show the existence of such a critical metric. This is achieved by the K-energy in the Calabi-Yau setting. \\

We define the Calabi-Yau K-energy $M_{\omega_0,\beta}$ as 
\begin{eqnarray*}& &M_{\omega_0,\beta}=\int_{M}\log (\frac{\omega^n_{\phi}}{e^{H_{\beta}}\omega_0^n})\frac{\omega^n_{\phi}}{n!}.
\end{eqnarray*}
Routine computation shows that 
\begin{eqnarray}\label{K-energy's derivative}\nonumber & &\frac{dM_{\omega_0,\beta}}{dt}
\nonumber \\&=&-\frac{1}{n!}\{\int_{M}\frac{\partial \phi}{\partial t}s_{\phi}\omega^n_{\phi}-2n\pi(1-\beta)\int_{D}\frac{\partial \phi}{\partial t}\omega^{n-1}_{\phi}\}
\nonumber \\&=&-\frac{1}{(n!)}\int_{M\setminus D}s_{\phi}\frac{\partial \phi}{\partial t}\omega^n_{\phi},
\end{eqnarray}
where $s_{\phi}$ is the scalar curvature of $\omega_{\phi}$.\\

Along the Calabi-Yau CKRF, we have 
\begin{equation}\label{Scalar curvature potential of Calabi-Yau CKRF}
\Delta_{\phi}\frac{\partial \phi}{\partial t}=-s_{\phi} \ \textrm{over}\ M\setminus D.
\end{equation}
Then (\ref{K-energy's derivative}) and (\ref{Scalar curvature potential of Calabi-Yau CKRF}) tell us 
\begin{equation}\label{Calabi-Yau: Decreasing of K-energy along CKRF}
\frac{dM_{\omega_0,\beta}}{dt}
=-\frac{1}{(n!)}\int_{M}|\nabla\frac{\partial \phi}{\partial t}|^2\omega^n_{\phi}\leq 0.
\end{equation}

By (\ref{Calabi-Yau: C11 bound}), (\ref{Calabi-Yau: bound on C2alpha beta norm of the potential}), and (\ref{Bound on the ocsillation of phi}), we see 
\begin{equation}\label{Calabi-Yau: Uniform bound on the K-energy along the flow}
|M_{\omega_0,\beta}(\omega_{\phi})|\leq C\ \textrm{over}\ [0,\infty).
\end{equation}
Since $\frac{dM_{\omega_0,\beta}}{dt}\leq 0$, then there exists a sequence $t_k\rightarrow \infty$ such that 
\begin{equation}\label{EC: derivative of K-energy subconverge to  0}
|\frac{dM_{\omega_0,\beta}}{dt}|_{t_k}\rightarrow 0.
\end{equation}
(\ref{Calabi-Yau: Decreasing of K-energy along CKRF}) and (\ref{EC: derivative of K-energy subconverge to  0}) imply 
\begin{equation}\label{Calabi-Yau: dphidt goes to 0}
\int_{M}|\nabla\frac{\partial \phi}{\partial t}|^2\omega^n_{\phi_{t_k}}\rightarrow 0.
\end{equation}
By the discussion at the end of Step 2 and (\ref{Calabi-Yau: dphidt goes to 0}), $\omega_{\phi_{t_k}}$ subconverges in $C^{\alpha,\beta}$ topology to a Ricci flat metric $\omega_{KE}$. At the point, we have already shown  the existence of a Ricci-flat metric.\\

Step 4: In this step we show the flow  converges to the unique $\omega_{KE}$ (obtained in the previous step) and the convergence  is exponential,  in the sense of (\ref{Calabi-Yau: Exponential convergence of the metric in the Calpha beta 1,1 topology.}). This is also straight forward by using the Calabi-Yau K-energy. Denote 
\[v= \frac{\partial \phi}{\partial t}-\frac{1}{Vol(M)}\int_{M} \frac{\partial \phi}{\partial t}\frac{\omega^n_{\phi}}{n!}.\]
Oboviously we have 
\begin{equation}\label{Calabi-Yau: Evolution equation for v}
\frac{\partial v}{\partial t}=\Delta_{\phi}v+\frac{1}{Vol(M)}\int_{M}|\nabla v|^2\frac{\omega^n_{\phi}}{n!}.
\end{equation}

Thus $v$ has zero average with respect to $\omega^n_{\phi}$ and Poincare inequality can be applied. 
By (\ref{Calabi-Yau: Bound on the C2alpha norm of dphidt}), we have 
\begin{equation}\label{Calabi-Yau: W1,2 norm v is bounded}
\int_{M}|\nabla v|^2\omega^n_{\phi}\leq C.
\end{equation}

From (\ref{Calabi-Yau: Uniform bound on the K-energy along the flow}) and (\ref{Calabi-Yau: Decreasing of K-energy along CKRF}) on the K-energy, for any $\epsilon>0$, there is a    $T_0$ large enough $T_0$  such that
\begin{equation}\label{Calabi-Yau: bound on the parabolic L^2 norm of v, the Ricci potential}
\int_{T_0}^{\infty}\int_{M}|\nabla v|^2\omega^n_{\phi}dt<\epsilon.
\end{equation}

Then using  parabolic Moser's iteration and (\ref{Calabi-Yau: bound on the parabolic L^2 norm of v, the Ricci potential}), by letting $\epsilon$ be small enough, we deduce 
\begin{equation}\label{Calabi-Yau: v could be smaller than 1/2}
|v|_{0,[T_0+1,\infty)}\leq \frac{1}{2}.
\end{equation}
Therefore, as in \cite{Cao}, consider
\[E=\frac{1}{2}\int_{M}v^2\omega^n_{\phi}.\]
Routine computation shows that 
\begin{equation}\label{Calabi-Yau: evolution of E }
\frac{\partial E}{\partial t}=-\int_{M}(1+v)|\nabla v|^2\omega^n_{\phi}.
\end{equation}
Combining (\ref{Calabi-Yau: v could be smaller than 1/2}) and (\ref{Calabi-Yau: evolution of E }) and the Poincare inequality in Remark 4.4 of \cite{WYQWF}, we compute
\begin{equation*}
\frac{\partial E}{\partial t}\leq -\frac{1}{2}\int_{M}|\nabla v|^2\omega^n_{\phi}\leq -\frac{C_{P}}{2}\int_{M} v^2\omega^n_{\phi}=-C_{P}E,\ \textrm{over}\ [T_0+1,\infty).
\end{equation*} 
Thus we obtain the exponential decay of the Dirichlet energy $E$:
\begin{equation}\label{Calabi-Yau:exponential decay of the Dirichlet energy }
E\leq Ce^{-C_{P}t}.
\end{equation}
Hence
\begin{equation}\label{Calabi-Yau:}
\int_{t-1}^{\infty}\int_{M} v^2\omega^n_{\phi}dt =2\int_{t-1}^{\infty}Edt\leq Ce^{-C_{P}t}.
\end{equation}
Using (\ref{Calabi-Yau:exponential decay of the Dirichlet energy }), by   integrating $\frac{\partial E}{\partial t}\leq -\frac{1}{2}\int_{M}|\nabla v|^2\omega^n_{\phi}$ from $t$ to $\infty$ we also  end up with a better decay estimate than (\ref{Calabi-Yau: bound on the parabolic L^2 norm of v, the Ricci potential}):
\begin{equation}\label{Calabi-Yau: Exponential decay estimate for the parabolic L^2 norm of v, the Ricci potential}
\int_{t}^{\infty}\int_{M}|\nabla v|^2\omega^n_{\phi}dt<Ce^{-C_{P}t}.
\end{equation}
Therefore  using (\ref{Calabi-Yau: Exponential decay estimate for the parabolic L^2 norm of v, the Ricci potential}) and the Poincare inequality to perform Moser's iteration to (\ref{Calabi-Yau: Evolution equation for v}),   we get a better decay estimate than 
(\ref{Calabi-Yau: v could be smaller than 1/2}):

\begin{equation}\label{Calabi-Yau: C0 norm of v decays exponentially}
|v|_{0,[t,\infty)}\leq Ce^{-C_{P}t}.
\end{equation}

By (\ref{Calabi-Yau: Evolution equation for v}), (\ref{Calabi-Yau: W1,2 norm v is bounded}), (\ref{Calabi-Yau: C0 norm of v decays exponentially}), (\ref{Calabi-Yau: Exponential decay estimate for the parabolic L^2 norm of v, the Ricci potential}), and Theorem 1.18 in \cite{CYW}, we obtain 
\begin{equation}\label{Calabi-Yau: Decay of the C2alpha beta norm of v}
|v|_{2,\alpha,\beta,[t,\infty)}\leq Ce^{-C_{P}t}.
\end{equation}
By the arguments in the proof of Proposition 2.2 in \cite{Cao} and (\ref{Calabi-Yau: Decay of the C2alpha beta norm of v}), we see  
\begin{equation}\label{Calabi-Yau: Decay of L1 norm of phi}
\int_{M}|\phi_{0}-\phi_{KE}|\omega^n_{D}|_{t}\leq Ce^{-C_{P}t}.
\end{equation}

Here $\phi_{KE}$ is normalized such that 
\[\frac{1}{Vol(M)}\int_{M} \phi_{KE}\frac{\omega^n_{D}}{n!}=0.\]

 Next, substract log of the equation 
 \begin{equation*}
(\omega_{D}+i\partial \bar{\partial}\phi_{KE})^n=e^{-h_{\omega_D}}\omega^n_{D}
\end{equation*}
from log of (\ref{potential equation of the Calabi-Yau CKRF}), we get the following linear equation
\begin{equation}\label{Calabi-Yau: Linearized equation of phi0-phiKE}
\underline{\Delta}(\phi_{0}-\phi_{KE})= v+a_{e}.
\end{equation}
where \[\underline{\Delta}=\int_{0}^{1}g^{i\bar{j}}_{b\phi_0+(1-b)\phi_{KE}}
\frac{\partial^{2}}{\partial z_i\partial\bar{z_j} }db\]
and $a_{e}\leq Ce^{-C_{P}t}$.\\

Then, finally, by (\ref{Calabi-Yau: Linearized equation of phi0-phiKE}), (\ref{Calabi-Yau: Decay of the C2alpha beta norm of v}), (\ref{Calabi-Yau: Decay of L1 norm of phi}), Theorem 1.18 in \cite{CYW}, and the Moser's iteration, we obtain our desired estimate 
\begin{eqnarray*}
& &|\phi_{0}-\phi_{KE}|_{2,\alpha,\beta, [t,\infty)}
\\&\leq &  C(|\phi_{0}-\phi_{KE}|_{L^1(M),[t-1,\infty)}+|v+a_{e}|_{\alpha,\beta, [t-1,\infty)})
\\&\leq & Ce^{-C_{P}t},
\end{eqnarray*}
which means the metric $\omega_{\phi}$ converges to $\omega_{KE}$ in the following sense
\begin{equation}\label{Calabi-Yau: Exponential convergence of the metric in the Calpha beta 1,1 topology.}
|\omega_{\phi}-\omega_{KE}|_{\alpha,\beta, [t,\infty)}\leq Ce^{-C_{P}t}.
\end{equation}

\end{proof}
\section{Appendix A: Liouville theorem when $\beta\leq \frac{1}{2}$.}

When $\beta<\frac{1}{2}$, Calabi's 3rd derivative estimate works in the conical case (see \cite{Brendle}). Though  Theorem \ref{thm Liouville} already settles down the Liouville theorem for all $\beta\in (0,1)$, it still might be interesting to present the following extremely short proof of the Liouville theorem when $\beta\leq \frac{1}{2}$.  
 
\begin{thm}(Weak Liouville Theorem)\label{thm rigidity of singularity model when beta less than 1} Suppose $\beta\leq \frac{1}{2}$ and  $\omega$ is a $C^{\alpha,\beta}$ conical K\"ahler  metric defined over $C^{n}$. Suppose 
\begin{equation}
\omega^{n}=\omega_{\beta}^{n},\ \frac{1}{C}\omega_{\beta}\leq \omega\leq C\omega_{\beta}\ \textrm{over}\ C^{n}.
\end{equation}
Then, there is a linear transformation $L$ which preserves $D$, such that  $\omega=L^{\star}\omega_{\beta}$.
\end{thm}
\begin{proof}{of Theorem \ref{thm rigidity of singularity model when beta less than 1}:}
Case 1: When $\beta=\frac{1}{2}$, the situation is very easy. Just consider the orbifold map $\mathfrak{T}:C^{n}\rightarrow C^{n}:$
\[\mathfrak{T}(w,u_1,..., u_{n-1})=(w^2,u_1,..., u_{n-1}).\]

Apparently, in this case, $\mathfrak{T}^{\star}\omega$ satisfies
\begin{equation}
(\mathfrak{T}^{\star}\omega)^{n}=\omega_{Euc}^{n},\ \frac{1}{C}\omega_{Euc}\leq \omega\leq C\omega_{Euc}\ \textrm{over}\ C^{n}\setminus \{z=0\},
\end{equation}
where $\omega_{Euc}$ is the regular Euclidean metric. Then using Proposition 16 in \cite{CDS3}, $\mathfrak{T}^{\star}\omega$ extends to a smooth positive $(1,1)$-form over $C^{n}$. Then $\mathfrak{T}^{\star}\omega=\underline{L}^{\star}\omega_{Euc}$, and $\mathfrak{T}^{\star}\omega$ is invariant under the deck transformation:
\[(z,...)\rightarrow (-z,...).\] Thus downstairs, we have
 $$\omega=L^{\star}\omega_{\frac{1}{2}},$$  
 where $L$ is a linear transformation which preserves $\{z=0\}$.

 Case 2: $\beta<\frac{1}{2}$. This is the case where we can do the 3rd-order estimate as Calabi, Yau, and Brendle.

We consider the scaling down again as 
\begin{equation}
\widehat{\phi}=R^{-2}\phi,\ \widehat{\omega}=R^{-2}\omega,\ \widehat{\omega}_{\beta}=R^{-2}\omega_{\beta}.
\end{equation}

$\widehat{\omega}_{\beta}$ is the standard conical model metric under 
the new coordinates $\widehat{z}=R^{-\frac{1}{\beta}}z$, $w_{i}=R^{-1}w_{i}$. Similarly we denote 
\[\widehat{S}=\widehat{\omega}^{i\bar{j}}\widehat{\omega}^{s\bar{t}}\widehat{\omega}^{p\bar{q}}\widehat{\phi}_{z_{i}, \bar{z}_{t},z_{p}}\widehat{\phi}_{\bar{z}_{j}, z_{s},\bar{z}_{q}}.\]
By formula $(2.7)$ in \cite{Yau}, we directly have  
\begin{equation}\label{equ laplace of second derivative controls 3rd derivative}
\Delta_{\widehat{\omega}}(\Delta_{\widehat{\omega}_{\beta}}\widehat{\phi})\geq   \frac{\widehat{S}}{K},\ K\geq 0.
\end{equation}
Then we multiply (\ref{equ laplace of second derivative controls 3rd derivative}) by a cutoff function $\eta^{2}$, and integrate  integration by parts with respect to $\omega$, we have 
\begin{equation}
\int_{C^{n}}\eta^2 \widehat{S}\widehat{\omega}^{n}\leq K\int_{C^{n}}(\Delta_{\widehat{\omega}}\eta^2)( \Delta_{\widehat{\omega}_{\beta}}\widehat{\phi})\widehat{\omega}^n.
\end{equation}
By the second order estimate, choose a proper cutoff function $\eta^2$ such that  $\eta=1$  in $B(1)$ and vanishes outside 
$B(2)$, we have 
\begin{equation}
\int_{B_{1}} \widehat{S}\widehat{\omega}^{n}\leq C|\Delta_{\widehat{\omega}_{\beta}}\widehat{\phi}|_{L^{\infty}B(1)}\leq C.
\end{equation}

Since our reference metric $\omega_{\beta}$ is flat, by the formula below formula (16) in \cite{Brendle}, we obtain
\begin{equation}\label{equ bochner for Calabi's 3rd order derivative}
\Delta_{\widehat{\omega}}\widehat{S}\geq 0.
\end{equation}
By Proposition 6.6 in \cite{Brendle}, we have $\widehat{S}\in L^{\infty}[B(2)]$. Thus, by Lemma \ref{lem Linfty solution is actually a weak solution},
$\widehat{S}$  is a weak subsolution to (\ref{equ bochner for Calabi's 3rd order derivative}).  Then, the Moser's iteration  as in Theorem 1.1 of Chap 4 in \cite{HanLin} is applicable.  We deduce

\begin{equation}
|\widehat{S}|_{L^{\infty}B(\frac{1}{2})}\leq \int_{B_{1}}\widehat{S}\widehat{\omega}^{n} \leq C.
\end{equation}

Then by rescaling, we have for $\omega$ that 
\begin{equation}
R^{2}|S|_{L^{\infty}B(\frac{R}{2})}\leq C.
\end{equation}
Then divide both hand sides by $R^{2}$, let $R\rightarrow \infty$, we have $S=0$ over $C^{n}$. $S=0$ implies $\omega$ is a covariant constant 
tensor with respect to $\omega_{\beta}$, then $\omega=L^{\star}\omega_{\beta}$, for some linear transformation $L$ preserving $D$.

The proof of Theorem \ref{thm rigidity of singularity model when beta less than 1} is thus completed.
\end{proof}

\section{Appendix B: Trudinger's Harnack inequality.}
In this section we work in the polar coordinates. 
\begin{lem}\label{lem Trudinger's Harnack inequality}  (Trudinger's Harnack inequality) Suppose $\omega$ is a weak conical metric. Suppose 
\begin{equation}\label{equ weak superharmonic inequality wrt omega}\Delta_{\omega}u\leq 0\ \textrm{in the weak sense in}\  B(R),\ u\in W^{1,2}[B(R)]\cap C^{2}[B(R)\setminus D],
\end{equation} and $u$ is nonnegative almost everywhere. Then for all $0<p< \frac{n}{n-1}$, we have 
\[R^{-\frac{2n}{p}}|u|_{L^{p}B(\frac{R}{2})}\leq C(p)\inf_{B(\frac{R}{4})}u, \]
where the $L^{p}$-norm $\textrm{is with respect to  the volume form of}\ \omega$.
\end{lem}
\begin{proof}{of Lemma :} Without loss of generality, we assume $R=1$.
Consider $\bar{u}=u+k$, $k>0$. Later we will let $k\rightarrow 0$. Consider the test function $\bar{u}^{-2}\varphi$. Then we apply the weak supersolution condition to  get 
\[-\int_{B(1)}\nabla_{\omega}\bar{u}\cdot \nabla_{\omega}(\varphi\bar{u}^{-2})\leq 0.\]
Hence
\[\int_{B(1)}(\nabla_{\omega}\frac{1}{\bar{u}})\cdot \nabla_{\omega}\varphi+2\int_{B(1)}|\nabla_{\omega}\bar{u}|\bar{u}^{-3} \varphi\leq 0.\]
Let $v=\frac{1}{\bar{u}}$, since $2\int_{B(1)}|\nabla_{\omega}\bar{u}|\bar{u}^{-3} \varphi\geq 0$, we end up with
\begin{equation}
\int_{B(1)}\nabla_{\omega}v\cdot \nabla_{\omega}\varphi\leq 0.
\end{equation}
This means $v$ is a positive weak-subsolution to $\Delta_{\omega}v\geq 0$! Since $\omega$ is a weak conical metric, the following holds by definition.
\begin{equation}\label{equ omega equivalent to the Euclidean metric in the polar coordinate}
\frac{g_{E}}{C}\leq \omega\leq Cg_{E}\ \textrm{over}\ B(1)\setminus D,
\end{equation}
where $
g_{E}\ \textrm{is the Euclidean metric in the polar coordinates}$.
Let $\varphi=\eta^{2}v^{p}$, by using Cauchy-Schwartz inequality,  we obtain
\begin{equation}\label{equ orlitz inequality wrt omega}
\frac{2p}{(p+1)^{2}}\int_{B(1)}\eta^2|\nabla_{\omega}v^{\frac{p+1}{2}}|^2\omega^{n}\leq \frac{4}{p}\int_{B(1)}|\nabla_{\omega}\eta|^2v^{p+1}\omega^{n}.
\end{equation}
By (\ref{equ omega equivalent to the Euclidean metric in the polar coordinate}) and (\ref{equ orlitz inequality wrt omega})
\begin{equation}\label{equ orlitz inequality wrt g_{E}}
\frac{2p}{(p+1)^{2}}\int_{B(1)}\eta^2|\nabla_{E}v^{\frac{p+1}{2}}|^2dvol_{E}\leq \frac{C}{p}\int_{B(1)}|\nabla_{E}\eta|^2v^{p+1}dvol_{E}.
\end{equation}
This is precisely the inequality which the Moser's iteration trick requires. Then from \cite{HanLin} Theorem 1.1 Chapter 4, we deduce for any $p>0$ that 
\[\sup_{B(\frac{1}{4})} v\leq C(p)|v|_{L^{p}B(\frac{1}{2})}.\]
Hence, 
\begin{equation}\label{equ inf control the 3 integrals}
\inf_{B(\frac{1}{4})}\bar{u}\geq (\int_{B(\frac{1}{2})}\bar{u}^{-p})^{-\frac{1}{p}}=\frac{(\int_{B(\frac{1}{2})}\bar{u}^{p}dvol_{E})^{\frac{1}{p}}}{(\int_{B(\frac{1}{2})}\bar{u}^{p}dvol_{E})^{\frac{1}{p}}(\int_{B(\frac{1}{2})}\bar{u}^{-p}dvol_{E})^{\frac{1}{p}}}.
\end{equation}
To apply the John-Nirenberg inequality, we need to verify the condition (7.51) in \cite{GT}, by the superharmonic equation in terms of $\omega$. Namely, the following claim is true. 
\begin{clm}\label{Claim logbaru satisfies the hypothesis of John-Nirenberg inequality} For any $p\in B(\frac{3}{4})$ and $r\leq 0.01$, we have \begin{equation}
\int_{B(r)}|\nabla_{E}\log \bar{u}|dvol_{E}
\leq Cr^{n}(\int_{B(r)}|\nabla_{E}\log \bar{u}|^2 dvol_{E})^{\frac{1}{2}} \leq Cr^{2n-1}.
\end{equation}
\end{clm}
To prove the claim, we apply  equation (\ref{equ weak superharmonic inequality wrt omega}) to the test function  $\eta^{2}\bar{u}^{-1}$, we get
\begin{equation}
-\int_{B(1)}2\eta\nabla_{\omega}\eta\cdot \frac{\nabla_{\omega}\bar{u}}{\bar{u}}\omega^{n}+\int_{B(1)}\eta^2\frac{|\nabla_{\omega}\bar{u}|^2}{\bar{u}^2}\omega^{n}\leq 0.
\end{equation}
By Cauchy-Schwartz inequality, we end up with 
\begin{equation}\label{equ condition verification of John Nirenberg}
\int_{B(1)}\eta^2|\nabla_{\omega}\log \bar{u}|^2\omega^{n}
\leq 16\int_{B(1)}|\nabla_{\omega}\eta|^2\omega^{n}.
\end{equation}
By (\ref{equ omega equivalent to the Euclidean metric in the polar coordinate}), we can transform the $W^{1,2}$-inequality in terms of $\omega$ to be in terms of $g_{E}$ again! Namely we have 
\begin{equation}\label{equ condition verification of John Nirenberg euclidean metric}
\int_{B(1)}\eta^2|\nabla_{E}\log \bar{u}|^2dvol_{E}
\leq C\int_{B(1)}|\nabla_{E}\eta|^2dvol_{E}.
\end{equation}
For any $p$ and $r$, we choose $\eta$ to be a cutoff function which is $1$
over $B_{p}(r)$,  $0$ over $C^{n}\setminus B_{p}(4r)$, and  $|\nabla_{E}\eta|\leq \frac{1}{r}$. Therefore, (\ref{equ condition verification of John Nirenberg euclidean metric}) implies 
\begin{equation*}
\int_{B(r)}|\nabla_{E}\log \bar{u}|dvol_{E}
\leq Cr^{n}(\int_{B(r)}|\nabla_{E}\log \bar{u}|^2 dvol_{E})^{\frac{1}{2}} \leq Cr^{2n-1}.
\end{equation*}
Thus Claim \ref{Claim logbaru satisfies the hypothesis of John-Nirenberg inequality} is proved. 
   
   Claim (\ref{Claim logbaru satisfies the hypothesis of John-Nirenberg inequality}) means $\log\bar{u}$ satisfies the hypothesis of the John-Nirenberg inequality in Theorem 7.21 in \cite{GT}. Then applying this theorem to $\log\bar{u}$, we obtain the following by exactly the argument in the last part of the proof of Theorem 8.18 in \cite{GT}.
   \[(\int_{B(\frac{1}{2})}\bar{u}^{p_{0}}dvol_{E})^{\frac{1}{p_{0}}}(\int_{B(\frac{1}{2})}\bar{u}^{-p_{0}}dvol_{E})^{\frac{1}{p_{0}}}\leq C
   \ \textrm{for some}\ p_{0}>0.\]
   Then by (\ref{equ inf control the 3 integrals}), we obtain
   \begin{equation}\label{equ inf control Lp0}
\inf_{B(\frac{1}{4})}\bar{u}\geq (\int_{B(\frac{1}{2})}\bar{u}^{p_{0}}dvol_{E})^{\frac{1}{p_{0}}}.
\end{equation}
Thus Lemma \ref{lem Trudinger's Harnack inequality} is already true for this particular $p_{0}$, Proof 2 of Theorem \ref{thm strong maximal principle} already goes through. 

To show Lemma \ref{lem Trudinger's Harnack inequality} is  true for all 
$0<p<\frac{n}{n-1}$, it suffies to show the following Claim holds.
\begin{clm}\label{clm reverse Holder inequality}For all $\frac{n}{n-1}>p> p_{0}$, we have
\[|\bar{u}|_{L^{p_0}[B(\frac{4}{5})]}\geq C(p)|\bar{u}|_{L^{p}[B(\frac{3}{4})]},\] 
where $ \textrm{the}\ L^{p}\ \textrm{is with respect to  the volume form of}\ g_{E}$.
\end{clm}
To prove the Claim, we should appeal to the superharmonicity (\ref{equ weak superharmonic inequality wrt omega}) again, and transform the $W^{1,2}$-type inequality with respect to $\omega$ to $W^{1,2}$-type inequality with respect to the Euclidean metric $g_{E}$. We start from applying (\ref{invariant solution upstairs}) to the test function $\bar{u}^{-a}\eta^2$. Then we end up with
\begin{equation}
a\int_{B(1)}|\nabla_{\omega}\bar{u}|^2\bar{u}^{-a-1}\eta^2\omega^n\leq 
\int_{B(1)}2\eta (\nabla_{\omega}\bar{u}\cdot\nabla_{\omega}\eta)(\bar{u}^{-a}). 
\end{equation}

Using Cauchy-Schwartz inequality and standard management again, we deduce
\begin{equation*}
\int_{B(1)}\eta^2|\nabla_{\omega}\bar{u}^{\frac{1-a}{2}}|^2\omega^{n}\leq \frac{(1-a)^2}{a^{2}}\int_{B(1)}|\nabla_{\omega}\eta|^2\bar{u}^{1-a}\omega^{n}.
\end{equation*}
Then, by (\ref{equ omega equivalent to the Euclidean metric in the polar coordinate}) again, we deduce 
\begin{equation}\label{equ reverse holder inequality}
\int_{B(1)}\eta^2|\nabla_{E}\bar{u}^{\frac{1-a}{2}}|^2dvol_{E}\leq\frac{C}{a^{2}}\int_{B(1)}|\nabla_{E}\eta|^2\bar{u}^{1-a}dvol_{E}.
\end{equation}
Thus we get  a reverse H\"older inequality with respect to $g_{E}$ from 
the reverse H\"older inequality with respect to $\omega$ again!. Apply exactly Step II of the proof in \cite{HanLin}, Claim \ref{clm reverse Holder inequality} holds.

Notice the $L^{p}$-norm with respect to $g_{E}$ is equivalent to the $L^{p}$-norm with respect to $\omega$. Now let $k\rightarrow 0$,     (\ref{equ inf control Lp0}) and  Claim \ref{clm reverse Holder inequality} directly imply Lemma \ref{lem Trudinger's Harnack inequality}. 
\end{proof}

Xiuxiong chen, Department of Mathematics, Stony Brook University,
NY, USA;\ \ xiu@math.sunysb.edu.\\

Yuanqi Wang, Department of Mathematics, University of California  at Santa Barbara, Santa Barbara,
CA,  USA;\ \ wangyuanqi@math.ucsb.edu.
 

\begin{thebibliography}{0}
 \bibitem{Ahlfors} L, Ahlfors. \emph{Complex Analysis}.  McGraw-Hill. 









\bibitem{Anderson}M,T, Anderson. \emph{Convergence and rigidity of manifolds under
Ricci curvature bounds.} Invent.math.102,429-445 (1990).
\bibitem{Aubin}T, Aubin. \emph{Nonlinear Analysis on Manifolds. Monge-Ampère Equations}. Grundlehren der mathematischen Wissenschaften.
Volume 252. 1982.
\bibitem{Berman}R, Berman.  \emph{A thermodynamic formalism for Monge-Ampere equations, Moser-Trudinger inequalities and Kahler-Einstein metrics}. Advances in Mathematics.
Volume 248, 25. November 2013, Pages 1254-1297.

\bibitem{Berman12}R, Berman. \emph{K-polystability of Q-Fano varieties admitting K\"ahler-Einstein
metrics.} arxiv1205.6214.
\bibitem{Brendle}S, Brendle. \emph{Ricci flat K\"ahler metrics with edge singularities}.  International Mathematics Research Notices 24, 5727--5766 (2013).

 \bibitem{Calabi58} E, Calabi.,\emph{Improper affine hyperspheres of convex type and a generalization of a theorem by K. J\"orgens.} Michigan .Math. J. 1958.
\bibitem{CGP} F, Campana; H, Guenancia; M, Paun. \emph{Metrics with cone singularities along normal crossing divisors and holomorphic tensor fields}. arXiv:1104.4879. To appear in Annales Scientifiques de l'ENS.
    \bibitem{Cao} H-D, Cao. \emph{Deformation of K\"ahler metrics to
K\"ahler-Einstein metrics on compact K\"ahler manifolds}.
Invent.Math.81(1985), no.2, 359-372.
\bibitem{CheegerColding} J, Cheeger; T,H, Colding. \emph{Lower Bounds on Ricci Curvature and the Almost Rigidity of Warped Products}.  Annals of Mathematics, 2nd Ser., Vol. 144, No. 1. (Jul., 1996), pp. 189-237.


\bibitem{CDS0} X-X, Chen; S, Donaldson; S, Sun. \emph{ K\"ahler-Einstein metrics and stability.} arXiv: 1210.7494. To appear in Int. Math. Res. Not (2013).
\bibitem{CDS1} X-X, Chen; S, Donaldson; S, Sun. \emph{ K\"ahler-Einstein metric on Fano manifolds, I: approximation of metrics with cone singularities}. arXiv1211.4566. To appear in JAMS.
\bibitem{CDS2} X-X, Chen; S, Donaldson; S, Sun. \emph{K\"ahler-Einstein metric on Fano manifolds, II:  limits with cone angle less than $2\pi$}.  arXiv1212.4714. To appear in JAMS.

\bibitem{CDS3} X-X, Chen; S, Donaldson; S, Sun. \emph{K\"ahler-Einstein metric on Fano manifolds, III:  limits with cone angle approaches  $2\pi$ and completion of the main proof}.   arXiv1302.0282. To appear in JAMS.

\bibitem{CWB}X-X, Chen; B, Wang. \emph{K\"ahler-Ricci flow on Fano manifolds (I)}. arXiv0909.2391. To appear in Journal of European Mathematical Society.
\bibitem{CWB1}X-X, Chen; B, Wang. \emph{On the conditions to extend Ricci flow(III)}. Int Math Res Notices (2012),
Vol.2012.


\bibitem{CYW}X-X, Chen; Y,Q, Wang. \emph{Bessel functions, heat kernel and the Conical K\"ahler-Ricci flow}. arXiv:1305.0255. 



\bibitem{Don} S,K, Donaldson.  \emph{K\"ahler metrics with cone singularities along a divisor}.  Essays in mathematics and its applications, 49–79, Springer, Heidelberg, 2012.

\bibitem{Evans}L,C, Evans. \emph{Partial differential equations}. Graduate Studies in Mathematics, Vol 19. AMS.
\bibitem{EGZ} P, Eyssidieux; V,  Guedj; A, Zeriahi. \emph{ Singular K\"ahler-Einstein metrics}. J. Amer. Math. Soc. 22 (2009), 607-639.
\bibitem{GP}Henri Guenancia, Mihai PÀòaun. \emph{Conic singularities metrics with perscribed Ricci
curvature: the case of general cone angles
along normal crossing divisors}. arXiv:1307.6375
\bibitem{GH}P, Griffith;J, Harris. \emph{Principles of Algebraic Geometry}. Wiley. 1994. 
\bibitem{GT} D,Gilbarg; N ,S, Trudinger.  \emph{Elliptic Partial Differential Equations of Second Order}. Springer.
\bibitem{Ha82}R, Hamilton. \emph{ Three-Manifolds with Positive Ricci Curvature}.\ J. Diff. Geom. 1982. Volume 17, Number 2 (1982), 255-306
.
\bibitem{Hormander} L, Hormander. \emph{$L^2$-estimates and existence theorems for the $\bar{\partial}$-operators}. Acta. Math. 113. (1965).
89-152.

\bibitem{HanLin}Q, Han; F,H, Lin. \emph{Elliptic Partial Differential Equations}. American Mathematical Soc. 2011.
 \bibitem{Jeffres}T, Jeffres. \emph{Uniqueness of K\"ahler-Einstein cone metrics}. Publ. Math. 44 (2000).


\bibitem{JMR}Jeffres, T; Mazzeo; R, Rubinstein. \emph{ K\"ahler-Einstein metrics with
edge singularities}. arXiv:1105.5216. To appear in Annals of Math.
\bibitem{LiYau} P, Li; S,T, Yau. \emph{On the parabolic kernel of the 
Schr\"odinger operator}. Acta. Mathematica.
July. 1986. Volume 156, Issue 1, pp 153-201.
\bibitem{Kolodziej} S, Kolodziej. \emph{H\"older continuity of solutions to the complex Monge-Ampere equation with the right-hand side in $L^p $: the case of compact K\"ahler manifolds.} Math Ann (2008), 379--386.
\bibitem{LSU}Ladyzenskaja; Solonnikov; Ural'ceva. \emph{Linear and quasi-linear equations of parabolic type}.  Translations of Mathematical Monographs 23, Providence, RI: American Mathematical Society.
\bibitem{LiSun} C, Li; S, Sun.\emph{ Conical K\"ahler-Einstein metric revisited}. arXiv1207.5011.

\bibitem{Lieberman} G, Lieberman. \emph{  Second order parabolic equations}. World Scientific, 1996.
\bibitem{LiuZhang}J,W, Liu; X, Zhang. \emph{The conical K\"ahler-Ricci flow on Fano manifolds}. arXiv:1402.1832.
\bibitem{MRS} R,Mazzeo;Y,Rubinstein;N,Sesum. \emph{Ricci flow on surfaces with conic singularities}. arXiv:1306.6688.





\bibitem{PhongSturm}D,H, Phong; J, Sturm. \emph{On stability and the convergence of the K√§hler-Ricci flow}. J. Differential Geom. Volume 72, Number 1 (2006), 149-168.

\bibitem{Po78}A,V,  Pogorelov. \emph{The multidimensional Minkowski problem.} Whasinton D.C. Winston, 1978. 
\bibitem{RS84}D, Riebesehl; F, Schulz.  \emph{A priori estimates and a Liouville theorem for complex Monge-Ampère equations.} Math. Z. 186 (1984), no. 1, 57–66. 

\bibitem{Siu}Y,T, Siu. \emph{Lectures on Hermitian-Einstein Metrics for Stable Bundles and K\"ahler-Einstein
Metrics}. Birkh¬®auser. 1987.
\bibitem{SongTian}J, Song; G, Tian. \emph{The K\"ahler-Ricci flow through singularities}. arXiv:0909.4898.








\bibitem{SongWang}J, Song; X, Wang. \emph{The greatest Ricci lower bound, conical Einstein
metrics and the Chern number inequality}.  arXiv1207.4839.
\bibitem{SongWeinkove}J, Song; B,  Weinkove. \emph{Contracting exceptional divisors by the K√§hler-Ricci flow}. arXiv:1003.0718. To appear in Duke Math. J.




\bibitem{Tian Yau1}G, Tian; S-T, Yau. \emph{ Complete K\"ahler Manifolds with Zero Ricci Curvature. I}. Journal of the American Mathematical Society, Vol. 3, No. 3. (Jul., 1990), pp. 579-609.

\bibitem{TZ1} G, Tian; X,H, Zhu. \emph{Convergence of K\"ahler-Ricci flow}.
J. Amer. Math. Soc. 20 (2007), no. 3, 675--699.


\bibitem{Wanglihe1}L,H, Wang. \emph{On the regularity theory of fully nonlinear parabolic equations}. Bull. Amer. Math. Soc. (N.S.) 22 (1990), no. 1, 107‚Äì114.
\bibitem{YuWang}Yu, Wang. \emph{A remark on $C^{2,\alpha}$-regularity of the complex
Monge-Ampere Equation}. arXiv:1111.0902.
\bibitem{WYQ}Y,Q, Wang. \emph{Notes on the $L^2$-estimates and regularity of parabolic equations over conical manifolds.} Unpublished work.
\bibitem{WYQWF}Y,Q, Wang. \emph{Smooth approximations of the Conical K\"ahler-Ricci flows}. arXiv:1401.5040.

\bibitem{Yao}C,J, Yao. \emph{Existence of Weak Conical K\"ahler-Einstein Metrics Along Smooth Hypersurfaces}. arXiv:1308.4307.
\bibitem{Yau} S-T, Yau. \emph{On the Ricci curvature of a compact K\"ahler manifold and the Complex Monge
Amp'ere equation I}. Comm. Pure Appl. Math. 31 (1978).
\bibitem{Ye}R, Ye. \emph{Sobolev Inequalities, Riesz Transforms and the Ricci Flow}. arXiv:0709.0512.
\bibitem{Yin} H, Yin. \emph{Ricci flow on surfaces with conical singularities}. Journal of Geometric Analysis.
October 2010, Volume 20, Issue 4, pp 970-995.
\bibitem{Yin2} H, Yin. \emph{Ricci flow on surfaces with conical singularities, II}.  arXiv:1305.4355.
\bibitem{QZhang1}Q, Zhang. \emph{A uniform Sobolev inequality under Ricci flow}. Int. Math. Res. Not. IMRN 2007, no. 17.
\bibitem{QZhang2}Q, Zhang. \emph{Bounds on volume growth of geodesic balls under Ricci flow}. arXiv:1107.4262.
\end{thebibliography}
\end{document}